\documentclass[a4paper]{article}

\usepackage[utf8]{inputenc} 
\usepackage[T1]{fontenc}    
\usepackage{hyperref}       
\usepackage{url}            
\usepackage{booktabs}       
\usepackage{amsfonts}       
\usepackage{nicefrac}       
\usepackage{microtype}      
\usepackage{xcolor}         

\usepackage[utf8]{inputenc}
 
\usepackage{xargs}                      

\usepackage[textsize=tiny]{todonotes}

\usepackage{hyperref}
		
\usepackage{fullpage,amsfonts,amsmath,amsthm,amssymb,mathtools}
\usepackage{tikz,color,algorithmic,algorithm,multirow,graphicx,caption,subcaption}
\theoremstyle{plain}
\newtheorem{theorem}{Theorem}
\newtheorem{corollary}[theorem]{Corollary}
\newtheorem{lemma}[theorem]{Lemma}
\newtheorem{remark}[theorem]{Remark}
\theoremstyle{definition}
\newtheorem{assumption}[theorem]{Assumption}

\newcommand{\R}{\mathbf{R}}
\newcommand{\oo}{\mathcal{O}}

\newcommand{\xs}{x^*}
\newcommand{\kk}{^{(k)}}
\newcommand{\kz}{^{(0)}}
\newcommand{\ko}{^{(1)}}
\newcommand{\kt}{^{(2)}}
\newcommand{\kit}{^{(k)}}
\newcommand{\kpo}{^{(k+1)}}
\newcommand{\kmo}{^{(k-1)}}

\newcommand{\eqdef}{:=}

\newcommand{\mat}[1]{\begin{bmatrix}#1\end{bmatrix}}
\usetikzlibrary{arrows}

\title{Gradient Descent and the Power Method:\\ Exploiting their connection to find the leftmost eigen-pair and escape saddle points}

\author{Rachael Tappenden and Martin Tak\'a\v{c}}

\graphicspath{{./fig/}{./}}

\begin{document}

\maketitle

\begin{abstract}
This work shows that applying Gradient Descent (GD) with a fixed step size to minimize a (possibly nonconvex) quadratic function is equivalent to running the Power Method (PM) on the gradients. The connection between GD with a fixed step size and the PM, both with and without fixed momentum, is thus established. Consequently, valuable eigen-information is available via GD.

Recent examples show that GD with a fixed step size, applied to locally quadratic nonconvex functions, can take exponential time to escape saddle points \cite{Du2017,Paternain2019}. Here, those examples are revisited and it is shown that eigenvalue information was missing, so that the examples may not provide a complete picture of the potential practical behaviour of GD. Thus, ongoing investigation of the behaviour of GD on nonconvex functions, possibly with an \emph{adaptive} or \emph{variable} step size, is warranted.

It is shown that, in the special case of a quadratic in $\R^2$, if an eigenvalue is known, then GD with a fixed step size will converge in two iterations, and a complete eigen-decomposition is available. 

By considering the dynamics of the gradients and iterates, new step size strategies are proposed to improve the practical performance of GD. Several numerical examples are presented, which demonstrate the advantages of exploiting the GD--PM connection.

  \end{abstract}

\section{Introduction}

This work concentrates on quadratic optimization problems of the form
\begin{equation}\label{probmin}
  \min_{x \in \R^n} f(x) \quad \text{where}\quad f(x) = \tfrac12 x^TAx - b^Tx,
\end{equation}
where $A\in \R^{n\times n}$ is symmetric, and $x,b \in \R^n$. If $A$ is Positive Definite (PD), then \eqref{probmin} is a strongly convex optimization problem. However, if $A$ is indefinite then \eqref{probmin} is nonconvex. Quadratic optimization problems arise in a plethora of situations, either directly, or perhaps as a subproblem within an iterative method. For example, many machine learning algorithms require a quadratic model to be minimized to determine an appropriate search direction within a trust region algorithm \cite{Nocedal06,Erway2020}, or within a Newton CG algorithm \cite{Nocedal06,Royer2020}. Moreover, nonconvex functions that are not quadratic, are often well approximated by a quadratic in the neighbourhood of a local minimizer or saddle point. Therefore, much can be learned by studying quadratic optimization problems.

One of the most well known algorithms in optimization is Gradient Descent (GD). Its popularity abounds because it is simple to use and understand, is widely applicable, and scales to high dimensional problems. Much can be said about the behaviour of GD on convex optimization problems, including characterisations of the convergence guarantees, and rates of convergence. However, it is also known to converge slowly in practice, particularly on ill-conditioned problems. 

Recently, partially motivated by the rise of applications in machine learning, attention has focused on the performance of GD on nonconvex problems. It is known that GD applied to nonconvex functions can find a stationary point in polynomial time \cite{Nesterov2004}, and that GD almost always escapes saddle points asymptotically \cite{Lee2016}, but there are other works \cite{Du2017,Paternain2019} that present nonpathalogical examples suggesting that GD can take exponential time to escape saddle points. Indeed, the work \cite{Du2017} asks the key question:
\begin{center}
`\emph{Does randomly initialized GD generally escape saddle points in polynomial time?}'
\end{center}
and they present examples demonstrating ``a strong \emph{negative} answer to this question'', \cite{Du2017}.

The current work revisits the examples presented in \cite{Du2017,Paternain2019} to show that, while the arguments in those works are sound (a \emph{fixed} step size is used for GD), important `free' eigenvalue information is available but not utilised, so further investigation of GD near saddle points is needed.

While studying the key question above, a connection between GD with a fixed step size and the PM, both with and without fixed momentum, is established. Hence, this connection shows that GD with a \emph{fixed} step size, applied to a (possibly nonconvex) quadratic function, can be adapted to provide an estimate of the leftmost eigenvalue-eigenvector pair of the Hessian at the cost of a couple of additional inner products. An algorithm called GD with EIGenvalue approximation, GD-EIG, is presented, which returns the minimizer of a quadratic function, and an approximate left eigen-pair. There are many potential uses for GD-EIG. For example, as an inner solver within an existing optimization algorithm,  the eigen information could be used to enrich a vanilla GD method, to directly find the leftmost eigen-pair of the Hessian, or as a stopping condition to classify a stationary point.

\subsection{Literature review}

Gradient Descent is ubiquitous in optimization and machine learning.  For convex optimization problems with a Lipschitz continuous gradient, GD converges at the (sublinear) rate $\oo(1/k)$, while for strongly convex problems, GD exhibits linear converge at the rate $1-1/\kappa$, where $\kappa$ denotes the condition number. These rates of convergence are suboptimal for strongly convex problems \cite{Nesterov2004}, and much effort has been made developing novel GD based algorithms that converge faster. 

One strategy that often helps improve the practical behaviour of GD is the heavy ball approach of Polyak \cite{Polyak1964}, where a `momentum' term is added at every iteration. While the iterates of GD often `zig-zag', momentum helps to damp this and pull the iterates along a `better' trajectory. While this approach can work well in practice, it can be difficult to tune the momentum parameter, and the algorithm is not guaranteed to converge in general.

Nesterov developed the concept of acceleration \cite{Nesterov1983,Nesterov2005,Nesterov2013}, and the accelerated gradient method is guaranteed to converge at the optimal linear rate $1-1/\sqrt{\kappa}$ for strongly convex functions. This algorithm works well in practice, but the strong convexity constant must be available because it is used in the algorithm. Many accelerated GD variants exist, including the following works \cite{Chen2017,Diakonikolas2019,Drusvyatskiy2018,Ghadimi2012,Jahani2021}, which also provide a certificate of optimality.

For strongly convex quadratic optimization problems, Conjugate Gradients (CG) \cite{Hestenes1952} is a successful algorithm, which is guaranteed to converge with rate $(\sqrt{\kappa}-1)/(\sqrt{\kappa}+1)$. Relevant to the current work is the connection between CG for solving problems of the form \eqref{probmin} when $A$ is positive definite, and the Lanczos method \cite{Lanczos1950} for finding the extremal eigenvalues (and eigenvectors) of $A$. In particular, if CG is applied to \eqref{probmin}, then approximations to the extremal eigenvalue(s) of $A$ can be recovered. Approximations to the corresponding eigenvectors are also available if the CG residual vectors are stored. (See also, for example, \cite[p.528]{Golub96}, \cite[p.307]{Demmel97}, \cite[p.221]{Saad2003} and \cite{Lanczos1952,Meurant2006}.)

Note that for CG, \eqref{probmin} must be strongly convex ($A$ must be PD). However, MINRES \cite{Paige1975} can be used even when $A$ is not PD, and the algorithm can also be used to generate eigen-information for $A$, i.e., there is a MINRES-Lanczos relationship. (See also \cite[p.494]{Golub96}.)

\subsection{Contributions}
The main contributions of this work are summarized below.
\begin{itemize}
  \item \emph{GD--PM Connection.} This work formalizes the connection between GD with a fixed step size and the Power Method (both with or without momentum). This relationship provides an analogue to the connection between CG and the Lanczos method.
  \item \emph{GD-EIG.} An algorithm called Gradient Descent with EIGenvalue Approximation (GD-EIG) is presented, which is simply GD with a fixed step size, combined with an explicit approximation of the leftmost eigenvalue of $A$. GD-EIG returns a solution to \eqref{probmin}, \emph{and} an estimate of the leftmost eigen-pair of $A$.
  \item \emph{Escaping saddle points.} Several examples from \cite{Du2017,Paternain2019}, examining the behaviour of GD with a fixed step size near saddle points, are revisited. It is shown that an approximation to the leftmost eigen-pair is implicitly available, and  strategies that use this information are suggested, providing potential ways of improving the practical performance of GD on these difficult nonconvex (`locally' quadratic) problems.
  \item \emph{Special case in $\R^2$.} Gradient descent with a fixed step size exhibits special behaviour in $\R^2$. Indeed, if the rightmost eigenvalue (the Lipschitz constant) is known, then a solution to \eqref{probmin}, and a full eigen-decomposition of $A$ is available in exactly 2 iterations.
  \item \emph{Eigen-enriched Gradient Descent.} New strategies are proposed to demonstrate how the eigen-information available via GD with a fixed step size might be used to improved the algorithm's behaviour. An algorithm called GD-EIG-Kick is presented, which is motivated by the GD-PM connection, and preliminary numerical experiments show that it has better practical performance compared with vanilla GD. 
\end{itemize}

\subsection{Preliminaries}
\label{Section_Preliminaries}
The following preliminaries and assumption are made throughout this work.
\begin{assumption}\label{ass:eigsA}
  The matrix $A\in \R^{n\times n}$ is symmetric. The eigen-pairs of $A$ are denoted by $(\lambda_i,v_i)$ for $i = 1,\dots,n$, with $\lambda_n \leq \cdots \leq  \lambda_1$, and $|\lambda_n|\leq |\lambda_1|$.
\end{assumption}

By the spectral theorem, the eigenvectors of $A$ form an orthonormal basis for $\R^n$. Thus, any vector $y\in \R^n$ can be represented in terms of this basis as $y = \sum_{i=1}^n (v_i^Ty)v_i.$
When $A$ is also nonsingular, the optimal solution to \eqref{probmin} is $x^* = \sum_{i=1}^n \tfrac1{\lambda_i}(v_i^Tb)v_i.$

The gradient of $f$ is 
\begin{equation}\label{gk}
  g(x) = A x - b. 
\end{equation}
(Usually the explicit dependence on $x$ is dropped, and the gradient is denoted by $g$).
The gradient of $f$ is $L$-Lipschitz continuous, i.e., for all $x,y\in\R^n$ it holds that
$\|Ax - Ay\|_2 \leq L \|x-y\|_2$ where $ L = \|A\|_2  = \lambda_1.$

Throughout this work, the convention $X^0 = I$, where $X \in \R^{n\times n}$ and $I$ denotes the identity matrix, is adopted (i.e., a matrix raised to the power zero is the identity matrix).

\subsection{Outline}

The paper is organised as follows. Section~\ref{S:GDandPM} presents Gradient Descent with momentum and the Power Method with momentum and establishes the connection between the two algorithms. Section~\ref{Section_Examples} revisits the examples in \cite{Du2017,Paternain2019} and shows that if eigenvalue information is taken into account, and the step size of GD is modified, then the practical behaviour of GD near saddle points may be better than previously suggested. The special case of GD in $\R^2$ is considered in  Section~\ref{Section_SpecialCaseR2}, and it is shown that convergence is possible in 2 iterations if the Lipschitz constant is known. Section~\ref{SectionEnrichedGD} considers practical aspects of GD given the established GD--PM connection. The dynamics of GD with a fixed step size are considered in Section~\ref{Section:GDdynamics}, and step size strategies are discussed in Section~\ref{Section_StepSize}. Section~\ref{Section_GDKick} describes the GD-EIG-Kick Algorithm, which uses the approximation to the leftmost eigenvalue to give GD a `kick' (a `long' step) towards the solution. Numerical experiments are presented in Section~\ref{Section_Numerical} to confirm that GD-EIG does recover an approximation to the leftmost eigen-pair in practice, and also to show the practical behaviour of GD-EIG-Kick. Concluding remarks are given in Section~\ref{Section_Conclusion}, and potential future research directions are suggested in Section~\ref{SectionNonconvexGDEIG}.

\section{Gradient Descent and the Power Method}\label{S:GDandPM}

In this section, Gradient Descent with momentum, and the Power Method with momentum, are discussed, and the connection between the two algorithms is formalized.

\subsection{Gradient Descent with Momentum}

Gradient Descent with Momentum (GDM), with a fixed step size $\alpha$ and fixed momentum term $\beta$, is presented in Algorithm~\ref{alg:GDM}. GDM is characterized by the iteration, for $k\geq 0$:
\begin{eqnarray}\label{xkpoGDmomentum}
  x\kpo = x\kk - \alpha g\kk + \beta(x\kk - x\kmo),
\end{eqnarray}
where $g\kk \eqdef g(x\kk)$.
In words, to generate the next iterate $x\kpo$, a step is taken from the current iterate $x\kk$ in the direction of the negative gradient of (fixed) size $\alpha$, and an adjustment term is added in the direction $x\kk-x\kmo$ scaled by (fixed) momentum parameter $\beta$. Setting $\beta = 0$ recovers standard GD with fixed step size.

\begin{algorithm}[ht!]
	\caption{Gradient Descent with Momentum (with fixed $\alpha$ and $\beta$)}
	\label{alg:GDM}
	\begin{algorithmic}[1]
		\STATE Initialization: Set $x\kz\in\R^n$, $x^{(-1)} = x\kz$, $\alpha\in (0,\tfrac1{\lambda_1}]$, $\beta\in[0,1]$.
		\FOR {$k=0,1,2,\dots$}
		\STATE  $g\kk = Ax\kk - b$
		\STATE  $x\kpo = x\kk - \alpha g\kk + \beta(x\kk - x\kmo)$
		\ENDFOR
	\end{algorithmic}
\end{algorithm}

It is convenient to define the following matrices:
\begin{equation}\label{H}
  H \eqdef (I-\alpha A),
\end{equation}
and
\begin{equation}\label{Hhat}
  \hat{H} \eqdef (1+\beta)I - \alpha A \overset{\eqref{H}}{=} H + \beta I.
\end{equation}
If $\beta = 0$ (no momentum) then $\hat{H} = H$. The following lemma shows how the gradient \eqref{gk} evolves when using GDM with fixed step-size $\alpha$, and fixed momentum parameter $\beta$.
\begin{lemma}\label{gradientHpower}
Let Assumption~\ref{ass:eigsA} hold, let $\alpha,\beta \geq 0$ be fixed, and let $H$ and $\hat{H}$ be defined in \eqref{H} and \eqref{Hhat}, respectively. Given an initial point $x\kz\in \R^n$, with $x^{(-1)} = x\kz$, at any iteration $k\geq 0$ of Gradient Descent with Momentum (Algorithm~\ref{alg:GDM}) applied to $f$ in \eqref{probmin},
  \begin{equation}\label{gradHm}
  g\kpo = \hat{H}g\kk - \beta g\kmo.
\end{equation}
\end{lemma}
\begin{proof}
The gradient for \eqref{probmin} is
  \begin{eqnarray}\label{gradientevolution}
  g\kpo &\overset{\eqref{gk}}{=}& Ax\kpo - b\notag \\
  &\overset{\eqref{xkpoGDmomentum}}{=}& A\left(x\kk - \alpha g\kk + \beta(x\kk - x\kmo)\right) - b\notag\\
  &=& g\kk - \alpha A g\kk+ \beta A (x\kk - x\kmo) \notag\\
  &\overset{\eqref{H}}{=}& Hg\kk + \beta A(x\kk - x\kmo)\notag\\
  &\overset{\eqref{gk}}{=}& Hg\kk + \beta (g\kk - g\kmo)\notag\\
  &\overset{\eqref{Hhat}}{=}& \hat{H}g\kk - \beta g\kmo.
\end{eqnarray}
\end{proof}

\begin{corollary}\label{gradientevolutioncorollary}
Let the conditions of Lemma~\ref{gradientHpower} hold with $\beta = 0$. Then for Gradient Descent (Algorithm~\ref{alg:GDM} with $\beta = 0$) applied to $f$ in \eqref{probmin},
\begin{equation}\label{gradH}
  g\kpo = Hg\kk= H^kg\kz.
\end{equation}
\end{corollary}

\begin{remark}[Computing the matrix-vector product efficiently]\label{EfficientHv}
  Note that $\hat{H}$ is never formed (nor is $H$). Instead, matrix vector products $\hat{H}v$ are computed as $u \gets Av$, followed by $\hat{H}v \gets (1+\beta)v - \alpha u$; only an oracle returning matrix vector products with $A$ is needed, i.e., the process is `matrix free'.
\end{remark}

\subsection{The Power Method with Momentum}

The Power Method (PM) can be used to approximate the dominant eigenvalue and corresponding eigenvector of a matrix $M\in \R^{n\times n}$: $(\nu_1,u_1)$. A variant of the PM is developed in \cite{Xu2018}, that can lead to faster approximation of the eigen-pair. The method, characterised by update (A) in \cite{Xu2018} (henceforth referred to as the PM with Momentum (PMM)), includes a momentum term with fixed momentum parameter $\beta$. The PMM is presented in Algorithm~\ref{alg:PMM}; setting $\beta = 0$ recovers the PM.
\begin{algorithm}[h!]
	\caption{Power Method with Momentum (PMM) (Rayleigh quotient approx)}
	\label{alg:PMM}
	\begin{algorithmic}[1]
		\STATE Initialization: Set $w^{(-1)}\equiv 0$, $w\kz\in\R^n$, momentum parameter $\beta\in[0,1]$.
		\FOR {$k=0,1,2,\dots$}
		\STATE  $w = Mw\kk - \beta w\kmo$ \label{w}
        \STATE  $\nu_1\kpo = (w^Tw\kk)/\|w\kk\|_2^2$\label{nu}
        \STATE $w\kpo = w/\|w\|_2$\label{unitw}
		\ENDFOR
	\end{algorithmic}
\end{algorithm}

In Algorithm~\ref{alg:PMM}, the Rayleigh quotient is used to approximate $\nu_1$, the dominant eigenvalue of $M$  (Step~\ref{nu}, $\nu_1\kpo \approx \nu_1$). The corresponding dominant eigenvector is approximated in Step~\ref{w}, $u_1 \approx w$. Step~\ref{unitw} is a normalization step to ensure that the approximate eigenvector does not grow too large, which can lead to numerical instabilities.

Convergence results for the PM can be found in many good textbooks (see, for example, Theorem~27.1 in \cite{Trefethen97}, and results in \cite{Demmel97,Golub96}). It is important to note that convergence of the PM requires the initial vector $w\kz$ to contain a component in the direction of the dominant eigenvector, i.e., $u_1^T w\kz\neq 0$. Convergence of the PMM was established in \cite{Xu2018}, and it was shown that the optimal choice for the momentum parameter $\beta$, for matrix $M$, is $\beta = (\nu_2(M))^2/4$. Thus, the PMM can be difficult to tune because a good approximation to $\nu_2(M)$, the second largest eigenvalue of $M$, is often unknown a priori.

\subsection{The Connection between GDM and PMM}

In GDM, the gradients evolve as in \eqref{gradHm} (recall Lemma~\ref{gradientHpower}),  which is equivalent to the evolution of the eigenvector approximation in Step~3 of the PMM (Algorithm~\ref{alg:PMM}) with $M = \hat H$ and $w\kz = g\kz$. Thus, GD with a fixed step size (either with or without fixed momentum) is implicitly running the PM(M) on the gradients. The connection between GDM and the PMM is now established.

Because the Power Method (with momentum) is an algorithm for approximating the dominant eigenvalue and corresponding eigenvector of a matrix, the observation that the gradients are computed via a PM (with momentum) iteration confirms that the gradients in GD are aligning in the direction of the dominant eigenvector of $\hat{H}$. Subsequently, one has access to approximations to the dominant eigen-pair of $\hat{H}$. The following lemma describes the spectrum, and eigenvectors, of $\hat H$.

\begin{lemma}[Eigenvalues of $\hat H$]\label{evalsofHhat}
  Let Assumption~\ref{ass:eigsA} hold, choose fixed $\alpha \in (0,\tfrac1{\lambda_1}]$ and $\beta \in [0,1]$. The eigen-pairs of $\hat H$ in \eqref{Hhat}, are $(\nu_i,v_{n-i+1})$, where
  \begin{equation}
    \nu_i = 1+ \beta -\alpha \lambda_{n-i+1},  \quad\text{ for all } \; i = 1,\dots,n.
  \end{equation}
\end{lemma}
\begin{proof}
By Lemma~\ref{ass:eigsA}, the eigen-pairs of $A$ are $(\lambda_i,v_i)$, for $i=1,\dots,n$, where $\lambda_n\leq \cdots \leq \lambda_1$. Thus, the eigenvalues of $\hat H = (1+\beta)I-\alpha A$ are $\nu_1 = 1+\beta -\alpha \lambda_n \geq  \cdots \geq \nu_n = 1+\beta-\alpha \lambda_1$. Finally, note that the ordering of the eigenvectors is reversed, so the $i$th eigenvalue of $\hat H$, $\nu_i$, is associated with the $(n-i+1)$th eigenvector of $A$, $v_{n-i+1}$.
\end{proof}
Lemma~\ref{evalsofHhat} shows that the dominant eigen-pair of $\hat H$ is $(\nu_1,v_n)$, where $\nu_1 = 1+\beta-\alpha \lambda_n$, while the leftmost eigen-pair of $A$ is $(\lambda_n,v_n)$, where $\lambda_n = \tfrac1{\alpha} (1+\beta-\nu_1)$. That is, Lemmas~\ref{gradientHpower}~and~\ref{evalsofHhat} show that the gradients in GD with a fixed step size are converging in the direction of the leftmost eigenvector of $A$, and an approximation to $\lambda_n$ is also readily obtained. Therefore, valuable curvature information related to $A$ is implicitly available via Gradient Descent with Momentum, with a fixed step size. Note also that, if $\lambda_n \geq 0$, then the spectral radius of $\hat H$ in \eqref{Hhat} is $\rho(\hat{H}) = 1+ \beta -\alpha \lambda_n \leq 1 + \beta$.

To the best of our knowledge, this is the first work to formally connect GDM and the PMM. However, note that mathematical expressions similar to \eqref{gradH} can be found in many textbooks on linear algebra and optimization. For example, (3.4) in \cite{vanderVorst2003} gives an expression for the relationship between the residuals (i.e., negative gradients) for Richardson's method, which is similar to \eqref{gradH}. Similarly, see the error reduction formula \cite[p.24]{vanderVorst2003}. See also the development in Chapter 2 of \cite{Greenbaum1997}, starting from expression (2.1) used with $M=\lambda_1 I$.
Despite similar relationships appearing in several places, we have not seen it explicitly written that GDM with fixed step size and momentum parameter is equivalent to the PMM acting on the gradients. Possibly this is because expressions such as \eqref{gradH} are often made as a initial argument, providing motivation for more sophisticated methods, which are shown to have better convergence properties than GD. (For example, expressing the residual at iteration $k+1$ in terms of a polynomial multiplied by the residual at iteration $k$, and showing that minimizing the polynomial with respect to different norms generates the iterates of various algorithms.)

\subsection{GDM with Eigenvalue Approximation Algorithm}

GDM and the PMM can be combined to give Algorithm~\ref{alg:GDEIG}.
\begin{algorithm}
	\caption{Gradient Descent with Eigenvalue Approximation (GD-EIG)}
	\label{alg:GDEIG}
	\begin{algorithmic}[1]
		\STATE Initialization: $x\kz\in\R^n$, $x^{(-1)} = x\kz$, $g\kz=Ax\kz-b = g^{(-1)}$ $\alpha\in(0,\tfrac1{\lambda_1}]$, $\beta\in[0,1]$.
		\FOR {$k=0,1,2,\dots$}
        \STATE $x\kpo = x\kk - \alpha g\kk + \beta(x\kk - x\kmo)$
        \label{step:3}
        \STATE $u\kk = \hat{H} g\kk$
        \label{step:4}
        \STATE $\nu_1\kk = (g\kk)^T u\kk/\|g\kk\|_2^2$
        \label{step:5}
        \STATE $g\kpo = u\kk - \beta g\kmo$
        \label{step:6}
        \STATE $\lambda_n\kk = (1+\beta-\nu_1\kk )/\alpha$
        \label{step:7}
		\ENDFOR
	\end{algorithmic}
\end{algorithm}

Gradient Descent with EIGenvalue Approximation, GD-EIG, is simply GDM plus an approximation to the leftmost eigenvalue of $A$. The eigen-pair approximation, $(\lambda_n,v_n) \approx ((1+\beta-\nu_1\kk )/\alpha,g\kk)$, is essentially `free', costing only a couple of inner products. If for any $k\geq 0$, $\nu_1\kk=1$, then $\lambda_n=0$, and therefore it is revealed that $A$ is singular. If for any $k\geq 0$, $\nu_1\kk > 1$, then $\lambda_n<0$, and therefore $A$ is indefinite. GD-EIG involves one matrix vector product per iteration (recall Remark~\ref{EfficientHv}), which is the same cost as for GDM. 

Note that the PMM includes a normalisation step (Step 4 in Algorithm~\ref{alg:PMM}), while GD, and GD-EIG, do not. For the PMM, the eigenvector approximation grows proportionally to the dominant eigenvector, so normalisation is included to stop the eigenvector approximation from growing too large, which could lead to numerical issues. On the other hand, for gradient descent with an appropriate $\alpha$ and $\beta$, the norm of the gradient is decreasing as the iterates move toward the optimal solution, and scaling the gradient is not desirable. (The norm of the gradient is often used as a stopping condition.)

Because GD-EIG is simply GDM with an approximation to the dominant eigenvalue of $\hat{H}$, convergence follows directly from the convergence properties of GDM. The Rayleigh quotient is used (Step~\ref{step:5}) because (for $\beta = 0$) ``the Rayleigh quotient is a quadratically accurate estimate of an eigenvalue'' \cite[p.204]{Trefethen97}. 

\begin{remark}\label{gzcomponentvn}
Note that convergence of the PM requires the initial vector to contain a component in the direction of the dominant eigenvector. Thus, to allow $v_n$ to be recovered, GD-EIG should also be initialized so that $v_n^Tg\kz \neq 0$, i.e., $g\kz$ must contain a component in the direction $v_n$, corresponding to the leftmost eigenvector.
\end{remark}

For a symmetric PD matrix, applying the PM ($\beta = 0$) to the shifted matrix $A - \lambda_1 I$ allows recovery of the smallest eigenvalue of $A$, i.e., the dominant eigenvalue of  $A-\lambda_1 I$ is $|\lambda_n - \lambda_1|$. 
Clearly, the rate of convergence of the Power Method applied to the shifted matrix is $\left| \frac{\lambda_{n-1}-\lambda_1}{\lambda_{n}-\lambda_1}\right|^{2k}$. For PD $A$, this matches the rate of GD-EIG because for $H$ in \eqref{H}, the rate of convergence of the PM is
\begin{eqnarray}\label{pmrate}
  \left| \frac{\nu_2}{\nu_1}\right|^{2k} = \left| \frac{1-\frac{\lambda_{n-1}}{\lambda_1}}{1-\frac{\lambda_{n}}{\lambda_1}}\right|^{2k} =\left| \frac{\lambda_1-\lambda_{n-1}}{\lambda_1-\lambda_{n}}\right|^{2k}.
\end{eqnarray}
The ratio in \eqref{pmrate} is smaller (i.e., convergence is faster) when the gap between $\lambda_{n-1}$ and $\lambda_n$ is larger. For a computable error bound on the accuracy of the eigenvalue approximation given by the Power Method  (Algorithm~\ref{alg:PMM} with $\beta=0$), see Theorem~8.1.13 and the comment on p.408 in \cite{Golub96}.

The step size for GD-EIG depends on $\lambda_1$, while the algorithm gives an approximation to $\lambda_n$, and there are many potential applications where this information could be useful. In particular, this algorithm could be used to directly estimate the eigen-pair $(\lambda_n,v_n)$, or to provide an estimate of the condition number of $A$. Note that several iteration complexity results exist for GD which involve an estimate of the condition number, so GD-EIG could be used to compute an estimate of the number of iterations required for convergence on-the-fly. An estimate of the leftmost eigenvalue could also be used as a second order convergence test to determine whether a stationary point is a local minimizer. On the other hand, once a satisfactory approximation to $\lambda_n$ is obtained, if $\lambda_n>0$, then one could switch to (a warm started) accelerated gradient method.

Furthermore, there are many algorithms that require information about the leftmost eigen-pair, and so GD-EIG could be used as an inner solver (see, for example, discussion of the `hard case' in \cite[Section~2.3]{Erway2020}, \cite[p.87]{Nocedal06}). Other potential applications include solving a sequence of equations with a slowly changing right hand side, or as a means of running a deflation technique \cite{Bellavia2013}.

\section{Behaviour of GD near saddle points via examples}
\label{Section_Examples}

Several works including \cite{Du2017} and \cite{Paternain2019}, present examples to show that GD can take exponential time to escape saddle points. They argue that vanilla GD with a fixed step size may not be the best choice for nonconvex optimization problems, and that other algorithms may be preferable in such cases. Here we revisit the examples to show that eigen-information could be used to help.

\subsection{Example of \cite{Paternain2019}}
An illustrative example is given in \cite{Paternain2019} which shows why gradient descent can escape from saddle points, but does so slowly. That work presents a Nonconvex Newton (NCN) algorithm, and the example is used to highlight the superiority of their proposed method. Thus, consider the following function:
\begin{equation}\label{nonconvex}
  f_{\sigma}(x) = \tfrac12 x_1^2 - \tfrac{\sigma}2 x_2^2 = \tfrac12 x^TAx,
\end{equation}
where $A = {\rm diag}(1,-\sigma)$. Clearly, the eigenvalues of $A$ are $\lambda_1 = 1$ and $\lambda_2 = - \sigma$ (it is assumed that $\lambda_1>0>\lambda_2$, i.e., $\sigma>0$). Consider gradient descent applied to minimize the function \eqref{nonconvex} using a fixed step size of $\alpha = 1/\lambda_1 = 1$ (no momentum is added so that $\beta = 0$). Take an initial point $x\kz = (x_1\kz,x_2\kz)^T$, and corresponding gradient $g\kz = Ax\kz = (x_1\kz, -\sigma x_2\kz)^T$. Then, for all $k\geq 1$ the iterates and gradient evolve, respectively, as
\begin{eqnarray}\label{itevolution}
  x_1\kit = 0,\qquad x_2\kit = (1+\sigma)^k x_2\kz.
\end{eqnarray}
and
\begin{eqnarray}\label{gevolution}
g_1\kit = 0,\qquad g_2\kit = -\sigma (1+\sigma)^k x_2\kz,
\end{eqnarray}
As $\sigma \to 0$, (i.e., as the problem becomes increasingly ill-conditioned), it takes longer to escape the saddle point.
In particular, as noted in \cite{Paternain2019}, this implies that GD with a fixed step size will take exponential time to escape from the saddle point with the rate $(1+\sigma)$.

On the other hand, the iterates of NCN (their proposed method) evolve as
\begin{eqnarray*}
  x_1\kit = 0,\qquad x_2\kit = 2^k x_2\kz.
\end{eqnarray*}
So NCN takes exponential time to escape the saddle point, but at the better rate of $2$, which is \emph{independent of} $\sigma$, (i.e., the rate is independent of problem conditioning). 

This example clearly demonstrates the drawbacks of vanilla gradient descent with a fixed step size. However, delving a little deeper into this example, and bringing in the connection with the Power Method, established previously in this work (i.e., consider GD-EIG rather than GD), reveals the following. The approximation to the dominant eigenvalue of $H$, for $k\geq 1$ is,
\begin{eqnarray}\label{evalapprox}
  \nu_1\kit = \frac{(g\kit)^THg\kit}{(g\kit)^Tg\kit} = \frac{(g\kit)^Tg^{(k+1)}}{(g\kit)^Tg\kit} = \frac{\sigma^2(1+\sigma)^{2k+1}(x_2\kz)^2}{\sigma^2(1+\sigma)^{2k}(x_2\kz)^2} = 1+\sigma.
\end{eqnarray}
Thus, after precisely 2 iterations of GD, the dominant eigenvalue of $H$ is recovered. Importantly, $\nu_1\ko > 1$, so it is immediately recognisable that the leftmost eigenvalue of $A$ is negative, and it can be explicitly recovered as $\lambda_2^{(1)} = \lambda_1(1-\nu_1^{(1)}) = -\sigma.$
The error
$\delta = \|\nu_1^{(1)} g^{(1)} - g^{(2)}\|^2  \overset{\eqref{gevolution}+\eqref{evalapprox}}{=} 0$ confirms that the approximation is exact, as is the corresponding eigenvector, $g^{(2)}\equiv v_2$.

In summary, after exactly 2 iterations of GD-EIG with an fixed step length $\alpha = 1/\lambda_1$, the leftmost eigen-pair of $A$ is recovered, and it is known that the original problem is nonconvex because $\lambda_2<0$ (so there is a saddle point). With vanilla gradient descent, this valuable curvature information is wasted because it is never explicitly computed. \emph{Can this eigen information be used to enrich gradient descent?}
This question will be investigated further in Section~\ref{SectionEnrichedGD}, but let us try to gain some intuition to address this question now using this example in $\R^2$. Taking an initial step of size $1/\lambda_1$ immediately reduced the problem to a 1-dimensional subspace ($g_1\kk = 0$ in \eqref{gevolution}). The iterate and gradient dynamics in \eqref{itevolution} and \eqref{gevolution} show that the step length is `too short' for this problem. Instead, what should the next iterate be? Let $\alpha\kt$ denote the potential step length. Then
\begin{eqnarray*}
  x_2^{(3)} = x_2\kt - \tfrac1{\alpha\kt}g_2\kt = (1+\sigma)^2x_2\kt - \tfrac1{\alpha\kt}(-\sigma(1+\sigma)^2x_2\kt) = \left(1+\tfrac{\sigma}{\alpha\kt}\right)(1+\sigma)^2x_2\kz.
\end{eqnarray*}
Setting $\alpha\kt = \lambda_2 = -\sigma$ leads directly to the saddle point $(0,0)$. On the other hand, choosing $\alpha\kt = -\lambda_2 = \sigma$, gives $x_2^{(3)} = 2 (1+\sigma)^2 x_2\kt$. Because $\sigma >0$, $2 (1+\sigma)^2> 2$, so that the `escape rate' is  better using this eigen-information. Thus, using a first-order method, enriched with curvature information, the step size can be adapted so that the iterates escape the saddle point faster than vanilla GD.

\subsection{Special Case: Dynamics of GD in $\R^2$}
\label{Section_SpecialCaseR2}

Motivated by the previous discussion, $\R^2$ is a special case from which additional insight into GD can be gained. This is investigated now, before returning to discuss the examples in \cite{Du2017}. 

Consider a symmetric matrix $A\in \R^{2\times 2}$, which has the eigendecomposition $A = V\Lambda V^T$, with $V = (v_1,v_2)$ being orthogonal, and $\Lambda = {\rm diag}(\lambda_1,\lambda_2)$. Let us study the behaviour of GD on problem \eqref{probmin} with the stated $A$ and $b\in \R^2$. Given $x\kz$, 
\begin{eqnarray}\label{g0R2}
g\kz = Ax\kz -b = \lambda_1(v_1^Tx\kz)v_1 + \lambda_2(v_2^Tx\kz)v_2-b.
\end{eqnarray}
It is helpful to express the step-length $\alpha$ in terms of the largest eigenvalue $\lambda_1$ as $\alpha = 1/c\lambda_1$, where (possibly unknown) $c\geq 1$ is a constant. (The restriction $c\geq 1$ ensures that $\alpha \in (0,1/\lambda_1]$.)\footnote{If $\lambda_1$ is known, then the choice $c=1$ should be used. If $\lambda_1$ is unknown, usually some overapproximation of $\lambda_1$ is used, and it is helpful to describe the known overapproximation $c\lambda_1$ (a product), in terms of the \emph{unknowns} $c$ \emph{and} $\lambda_1$ .} Noting that any vector $y\in \R^n$ can be represented in terms of an orthogonal basis of eigenvectors as $y = \sum_{i=1}^n (v_i^Ty)v_i$, one has
\begin{eqnarray}\label{x1R2}
x^{(1)} &=& x\kz - \alpha g\kz\notag\\
&\overset{\eqref{g0R2}}{=}& (v_1^Tx\kz)v_1 + (v_2^Tx\kz)v_2 - \tfrac{1}{c\lambda_1}\left(\lambda_1(v_1^Tx\kz)v_1 + \lambda_2(v_2^Tx\kz)v_2-b\right)\notag\\
&=& (1 - \tfrac{1}{c})(v_1^Tx\kz)v_1 +(1- \tfrac{\lambda_2}{c\lambda_1})(v_2^Tx\kz)v_2  + \tfrac{1}{c\lambda_1}b
\end{eqnarray}
\paragraph{Case I: $c=1$ (i.e., $\lambda_1$ is known).} If $\lambda_1$ is known then the choice $c=1$ is made. The step-length is $\alpha = 1/\lambda_1$ and the first term in \eqref{x1R2} involving $v_1$ disappears, giving
\begin{eqnarray}\label{x1R2c1}
x^{(1)} =  (1- \tfrac{\lambda_2}{\lambda_1})(v_2^Tx\kz)v_2  + \tfrac{1}{\lambda_1}b.
\end{eqnarray}
Thus, $x^{(1)} - \tfrac{1}{\lambda_1}b $ lies in the 1-dimensional subspace of $\R^2$ spanned by $v_2$, i.e., $x^{(1)} - \tfrac{1}{\lambda_1}b $ and $v_2$ are co-linear. A unit eigenvector $v_2$ is recovered immediately as
\begin{eqnarray}\label{v2R2x1}
v_2 = (x^{(1)} - \tfrac{1}{\lambda_1}b) /\|x^{(1)} - \tfrac{1}{\lambda_1}b \|_2,
\end{eqnarray}
with corresponding eigenvalue $\lambda_2 = v_2^TAv_2/(v_2^Tv_2)$. If one possesses the eigenvector $v_1$, and if $\lambda_2 \neq 0$, then the stationary point is recovered as
\begin{eqnarray}\label{xsR2}
\xs = \tfrac1{\lambda_1}(v_1^Tb)v_1 + \tfrac1{\lambda_2}(v_2^Tb)v_2.
\end{eqnarray}
If not (i.e., if $v_1$ is unknown), then from \eqref{x1R2c1}, the gradient at $x^{(1)}$ is
\begin{eqnarray}\label{g1R2}
g^{(1)} &=&  Ax^{(1)} -b \notag\\ 
&=& (1- \tfrac{\lambda_2}{\lambda_1})(v_2^Tx\kz)Av_2  + \tfrac{1}{\lambda_1}Ab - b\notag\\
&= & \lambda_2(1- \tfrac{\lambda_2}{\lambda_1 })(v_2^Tx\kz)v_2 -Hb.
\end{eqnarray}
If $g^{(1)} + Hb = 0$ then from \eqref{g1R2}, it must hold that $\lambda_2 = 0$. Otherwise, one can also recover the (unit) eigenvector $v_2$ as
\begin{eqnarray}\label{v2R2viag1}
v_2 = (g^{(1)} + Hb)/\|g^{(1)} + Hb\|_2.
\end{eqnarray}
Assuming $\lambda_2 \neq 0$, taking a step of size $\alpha^{(2)} = 1/\lambda_2$ gives
\begin{eqnarray*}
x^{(2)} &=& x^{(1)} - \tfrac{1}{\lambda_2}g^{(1)}\\
&\overset{\eqref{x1R2c1}+\eqref{g1R2}}{=}& \tfrac{1}{\lambda_1}b + \tfrac{1}{\lambda_2}Hb\\
&=& \tfrac{1}{\lambda_1}\left((v_1^Tb)v_1 +(v_2^Tb)v_2 \right) + \tfrac{1}{\lambda_2}(1-\tfrac{\lambda_2}{\lambda_1})(v_2^Tb)v_2\\
&=& \tfrac{1}{\lambda_1}(v_1^Tb)v_1 + \tfrac{1}{\lambda_2}(v_2^Tb)v_2,
\end{eqnarray*}
which is the stationary point \eqref{xsR2}. 
Now $g^{(2)} = 0$, so the stationarity of $x^{(2)}$ is confirmed.

If the eigenvalues are both positive, then the stationary point is a minimizer, while if they have opposite signs it is a saddle point. Thus, in $\R^2$, if $\lambda_1$ is known, precisely 2 steps of a gradient based method (i.e., gradient descent with step-lengths $\alpha^{(1)}=1/\lambda_1$ and $\alpha^{(2)}=1/\lambda_2$) are required to find a stationary point and determine it's nature. Of course, if it is found using \eqref{v2R2x1} that $\lambda_2 <0$, then the information above can be used to escape the saddle point more quickly, rather that stepping directly toward it.

If $\lambda_2 = 0$ then $A$ is rank deficient and the saddle point is non-strict. Substituting $\lambda_2 = 0$ into \eqref{x1R2c1} gives
$x^{(1)} =  (v_2^Tx\kz)v_2  + \tfrac{1}{\lambda_1}((v_1^Tb)v_1 +(v_2^Tb)v_2 )$, because $b = (v_1^Tb)v_1 +(v_2^Tb)v_2 $.
Now
$g^{(1)} =  (v_2^Tb)v_2$.
Thus, if $b\in {\rm range}(A)$ then $v_2^Tb =0$ so that $g^{(1)} = 0$, which confirms that $x^{(1)}$ is a minimizer of $f$. On the other hand, if $b\not\in {\rm range}(A)$, $(v_2^Tb)v_2 \;\;(\neq 0)$ is the projection of $b$ onto $v_2$, so that again $x^{(1)}$ is a minimizer of $f$, and $\|g^{(1)}\|_2 = v_2^Tb$.

\paragraph{Case II: $c>1$ (i.e., $\lambda_1$ and $c$ are unknown).} If $c>1$, simply continuing with the step-length $\alpha=1/c\lambda_1\in(0,1/\lambda_1]$ shows that the dynamics of the gradient can be expressed as
\begin{eqnarray*}
g\kk = (1-\tfrac1c)^k(v_1^T(\lambda_1x\kz - b))v_1 + (1-\tfrac{\lambda_2}{c\lambda_1})^k(v_2^T(\lambda_2x\kz-b))v_2.
\end{eqnarray*}
The eigenvalue approximation is
\begin{eqnarray}\label{evalapprox}
  \nu_1\kk = \frac{(g\kit)^THg\kit}{(g\kit)^Tg\kit} = \frac{(1-\tfrac1c)^{2k+1}(v_1^T(\lambda_1x\kz - b))^2 + (1-\tfrac{\lambda_2}{c\lambda_1})^{2k+1}(v_2^T(\lambda_2x\kz-b))^2}{(1-\tfrac1c)^{2k}(v_1^T(\lambda_1x\kz - b))^2 + (1-\tfrac{\lambda_2}{c\lambda_1})^{2k}(v_2^T(\lambda_2x\kz-b))^2}
\end{eqnarray}

Because $(1-\tfrac1c)<(1-\tfrac{\lambda_2}{c\lambda_1})$, the first term on the numerator and denominator go to zero faster than the second term on the numerator and denominator. As soon as $\nu_1^{(K)}\approx \nu_1^{(K+1)}$ for some $K$, it is known that $\nu_1^{(K)} \approx (1-\tfrac{\lambda_2}{c\lambda_1})$ and the description from Case I can be applied to approximate $(\lambda_2,v_2)$, and either head toward, or escape from, the stationary point as appropriate.

\subsection{Example of \cite{Du2017}}

Equipped with the information from Section~\ref{Section_SpecialCaseR2} --- that GD will converge in exactly 2 iterations in $\R^2$ if available eigen-information is utilized --- we return to the examples in~\cite{Du2017}. 

The following example is taken directly from \cite{Du2017}. Consider a two-dimensional function $f$ with a strict saddle point at (0,0). Suppose that inside the neighbourhood $U=[-1,1]^2$ of the saddle point, the function is locally quadratic $f(x_1,x_2) = x_1^2 - x_2^2$. (Note that the eigenvalues are $\lambda_1 = 2$ and $\lambda_2 = -2$.) For GD with $\alpha = 1/4$ the coordinates are updated as
\begin{eqnarray*}
  x_1\kpo = \tfrac12 x_1\kk \quad \text{and} \quad x_2\kpo = \tfrac32 x_2\kk.
\end{eqnarray*}
The authors of \cite{Du2017} argue that if GD is initialized uniformly within the exponentially thin band $[-1,1] \times [-(\tfrac32)^{-\exp(1/\epsilon)},(\tfrac32)^{-\exp(1/\epsilon)}]$, (the width is $2(\tfrac32)^{- \exp (1/\epsilon)}$) then GD requires at least $\exp(1/\epsilon)$ iterations to get out of the neighbourhood $U$ and escape the saddle point.

The following comments can be made about this example. First, note that the step length is half what it could be to ensure convergence. That is, $\alpha = \tfrac14 = \tfrac1{2\lambda_1} = \tfrac1{c\lambda_1}$ with $c=2$, whereas the arguments in Section~\ref{Section_SpecialCaseR2} suggested to use the step size $\alpha =\tfrac1{\lambda_1} = \tfrac12$. It is immediately recognisable (for this very special case where the eigenvectors align in the coordinate directions) that the step is too short, because $x_1\ko \neq 0$, $g_1\ko \neq 0$, and because
\begin{eqnarray*}
  g_1\kpo = 2\cdot\left(\tfrac12\right)^k x_1\kz \quad \text{and} \quad g_2\kpo = 2\cdot\left(\tfrac32\right)^k x_2\kz.
\end{eqnarray*}
The eigenvalue approximation is
\begin{eqnarray*}
  \nu_1\kk = \frac{(g\kk)^Tg\kpo}{(g\kk)^Tg\kk} = \frac{(1/2)^{2k+1}(x_1\kz)^2 + (3/2)^{2k+1}(x_2\kz)^2}{(1/2)^{2k}(x_1\kz)^2 + (3/2)^{2k}(x_2\kz)^2}.
\end{eqnarray*}

Before presenting several numerical experiments, note that the suggested initialization band for the component $x_2\kz$ is so thin that for even large values of $\epsilon$, one has $x_2\kz = 0$. In particular, MATLAB reports that $x_2\kz = 0$ when $\epsilon = 0.1$ (and $x_2\kz = 10^{-27}$ when $\epsilon = 0.2$). The following examples were all initialized at $x\kz = (1,(\tfrac32)^{-\exp(1/\epsilon)})$ for either $\epsilon = 0.1$ (i.e., $x_2\kz = 0$) or $\epsilon = 0.5$ (i.e., $x_2\kz \neq 0$). Recall that $\lambda_2\kk = (1-\nu_1\kk)/\alpha$ and $\delta\kk = \|Ag\kk - \lambda_2\kk g\kk\|_2$.

\paragraph{Numerical experiment 1: $\epsilon = 0.1$.} A numerical experiment is performed on the stated example with $\epsilon = 0.1$, and GD-EIG uses a step size of $\alpha = 1/4$. The iterates are $x\kz = (1,0)^T$, and for $k=1$ GD-EIG gives  $x\ko = (0.5,0)^T$, $\nu_1\ko = 0.5$, $\lambda_2\ko = -2$ and $\delta\ko = 0$. Because $\delta\ko = 0$, it is known that GD-EIG has found an exact eigenvalue, and that $g\ko$ points in the direction of the corresponding eigenvector. Furthermore, it is clear that $\alpha$ is too short, and taking the step $x^{(2)} = x\ko - \tfrac1{\lambda_2} g\ko = (0,0)^T$, i.e., the saddle point has been reached. The saddle point \emph{cannot be escaped}, because there is no `escaping direction'. Recall Remark~\ref{gzcomponentvn}, which explains that this is because $v_2^Tg\kz = 0$. 

Because a stationary point is found, one could stop here. Alternately, should one wish to learn the nature of the stationary point, consider the following. Given that an eigen-pair is now known, choose a new random point, $\tilde x\kz$ say, (possibly close to the solution/origin) that is \emph{not co-linear} with the known eigenvector. Then running 2 iterations of GD-EIG from $\tilde x\kz$ with the step size $\alpha = |1/\lambda_2| = 1/2$ will give the complete information (see also numerical experiment 2).

\paragraph{Numerical experiment 2: $\epsilon = 0.5$.} Here, a numerical experiment is performed on the stated example, where $\epsilon = 0.5$, and where GD-EIG uses the `optimal' step size $\alpha = 1/\lambda_1 = 1/2$ (i.e., assume that $\lambda_1$ is \emph{known}). The iterates are: $x\kz = (1,0.05)^T$, $x\ko = (0,0.1)^T$ with $\delta\ko \neq 0$, and $x^{(2)} = (0,0.1999)^T$ with $\nu_1^{(2)} = 2 \;(>1)$, $\lambda_2^{(2)} = -2$, and $\delta\ko = 0$. Thus, after precisely 2 iterations of GD-EIG with $\alpha = 1/2$, both eigenvalues are known, and the eigenvector corresponding to the negative eigenvalue (i.e., an escaping direction) is known. Taking the step $\alpha^{(2)} = 1/\lambda_2 = -1/2$ leads directly to the saddle point.

\paragraph{Numerical experiment 3: $\epsilon = 0.5$.} Here a numerical experiment is performed on the stated example to illustrate the behaviour of GD, where $\epsilon = 0.5$. Running GD-EIG with $\alpha = 1/4$ gives the iterates reported in Table~\ref{tab:my_labelAA}.
\begin{table}[h!]
    \caption{Behaviour of GD-EIG on the above example with $\epsilon = 0.5$ and $\alpha = 1/4$.}
    \label{tab:my_labelAA}
    \centering
    \begin{tabular}{c|c|c|c|c|c}
       $k$      &  0 & 1 & 2 & 3 & 4\\
       \hline
       $x_1\kk$ & 1 & 0.5 & 0.25 & 0.125 & 0.0625\\
       $x_2\kk$ & 0.05 & 0.075 & 0.1125 & 0.1687 & 0.2531\\
       \hline
       $\nu_1\kk$ & --- & 0.5025 & 0.5220 & 0.6683 & 1.1456\\
       \hline
       $\lambda_2\kk$ & --- & 1.9900 & 1.9120 & 1.3267 & -0.5823\\
       \hline
       $\delta\kk$ & --- & 0.5984 & 0.8811 & 1.1350 & 0.7868\\
    \end{tabular}

\end{table}

Notice that after 4 iterations, the approximation $\nu_1^{(4)}>1$, so it is known that $A$ has a negative eigenvalue. Therefore, it is also known that $g^{(4)} \;(= (0.1250,-0.5061)^T) $ is a direction of negative curvature. (Note that $\delta^{(4)} \neq 0$, $\lambda_2\kk \neq \lambda_2$ and $g^{(4)} \neq v_2$.) Given this information, it is appropriate to consider whether to employ a different strategy, rather than continuing with GD-EIG under the given settings. Possible strategies include: (1) Taking a step of size $\alpha^{(4)} = -1/\lambda_2^{(4)}$. (This would give $x^{(5)} = x^{(5)} - \alpha^{(4)} g^{(4)} = (-0.1522,1.1221)^T$); (2) Performing a forward tracking line search along the negative curvature direction; or (3) Continuing with GD-EIG until $\delta\kk \approx 0$. (When $k=6$, $\lambda_2^{(6)} = -1.9731$ and $\delta^{(6)} = 0.1279$, while for $k=10$, $\lambda_2^{(10)} = -2.0000$ and $\delta^{(10)} = 0.0078$.)

\begin{remark}
The paper \cite{Du2017} also considers examples where GD is initialized `far away', and where multiple saddle points need to be escaped sequentially. The arguments above show that, so long as $v_2^Tg\kz \neq 0$, GD-EIG requires 2 iterations to find the stationary point and reveal its nature, i.e., it does not matter if $x\kz$ is `far away', but it does matter that $v_2^Tg\kz \neq 0$.
\end{remark}

\section{Practical considerations for GD}
\label{SectionEnrichedGD}

The previous sections established that GD with a fixed step size is connected to the PM acting on the gradients via a `shifted' matrix $\hat H$, so that valuable eigen-information is available when running GD on a quadratic optimization problem. This section investigates how to utilize this curvature information. The dynamics of the gradient are studied, the choice of step sizes is discussed, and then a new eigen-enriched GD variant, GD-Kick, is presented. Note that this section only considers the case $\beta = 0$ (no momentum). 

\subsection{Dynamics of GD}\label{Section:GDdynamics}

The dynamics of the gradient for GD with a fixed step size are presented now. Let $A\in\R^{n\times n}$ have the eigen-decomposition $A = V\Lambda V^T$, where $V = \mat{v_1 & \cdots &v_n}$ and $\Lambda = {\rm diag}(\lambda_1,\dots,\lambda_n)$. Because $A$ is symmetric (Assumption~\ref{ass:eigsA}), $V$ is orthogonal, and the eigenvectors form an orthonormal basis for $\R^n$. For simplicity, suppose that the eigenvalues are distinct and nonzero. The eigen-decomposition is $H$ is
\begin{eqnarray}\label{eigdecompH}
H = I-\alpha A = I -\alpha V\Lambda V^T = V(I -\alpha \Lambda )V^T =  V\Lambda_HV^T,
\end{eqnarray}
where $\Lambda_H = {\rm diag}(1-\alpha\lambda_1,\dots,1-\alpha\lambda_n)$ and $A$ and $H$ share the common eigenvectors given in $V$. Then, for any iteration $k\geq 0$, the gradient is
\begin{eqnarray}\label{gdynamicsnocycling}
  g\kk \overset{\eqref{gradH}}{=} H^k g\kz 
 \overset{\eqref{eigdecompH}}{=} V\Lambda_H^kV^T g\kz
   = \sum_{i=1}^n (1-\alpha \lambda_i)^k(v_i^Tg\kz)v_i.
\end{eqnarray}
The dynamics of the gradient \eqref{gdynamicsnocycling} are revealing. Suppose that $v_i^Tg\kz \neq 0$ $\forall i$, and suppose that $\lambda_n>0$ (i.e., let $A$ be PD).
Then, for $g\kk \to 0$  as $k\to \infty$, all the coefficients $(1-\alpha \lambda_i)^k$ must tend to zero. If $\alpha = 1/\lambda_1$, then $(1-\alpha \lambda_1)^1 = 0$, and  $(1-\alpha \lambda_i) = (1-\tfrac{\lambda_i}{\lambda_1})<1$  for $i=2,\dots,n$, so $(1-\alpha \lambda_i)^k \to 0$.
So, after 1 iteration of GD with the step length $\alpha = 1/\lambda_1$, the component of the gradient in the direction of $v_1$ is eradicated, and the gradient lies in the $n-1$-dimensional subspace spanned by the remaining eigenvectors $\{v_2,\dots,v_n\}$. The term $(1-\alpha\lambda_n)^k$ tends to zero the most slowly, confirming from a different viewpoint, that the gradient aligns in the direction of the leftmost eigenvector as $k\to \infty$.

Following this line of argument, if one knew the eigenvalues $\lambda_1,\dots,\lambda_n$, taking precisely $n$ steps of gradient descent, but with steplength $1/\lambda_k$ for iterations $k=1,\dots,n$, will lead to a zero gradient.\footnote{In fact, in exact arithmetic, the order of the step-lengths does not matter, only that exactly one step of size $1/\lambda_k$ for each $k=1,\dots,n$ is taken. However, for numerical reasons, the order $1/\lambda_1,\dots,1/\lambda_n$ is sensible.} That is, \eqref{gdynamicsnocycling} shows that once an eigen-component has been cancelled out from the gradient, it will never re-enter the gradient (in exact arithmetic) so that taking $n$-steps, each of which is in the direction of the negative gradient and with each step size being the reciprocal of each eigenvalue, will lead to the stationary point. (This is also true when $A$ is indefinite.)

If $A$ is indefinite, then, for all $\lambda_j < 0$, $(1-\tfrac{\lambda_j}{\lambda_1})> 1$, so that continuing with the fixed step size $\alpha = 1/\lambda_1$ will cause the coefficients in \eqref{gdynamicsnocycling} corresponding to negative eigenvalues to grow, and the iterates will escape the saddle point eventually.

\subsection{Choice of step length}\label{Section_StepSize}

It is now clear that for $k>1$, the step length $\alpha = 1/\lambda_1$ is pessimistic. 
Instead, assuming that $\lambda_1$ is known, consider the following  `smart' initialization strategy for GD with a fixed step size. Choose a random point $x^{(-1)}\in \R^n$ and then set the initial point to be $x\kz = x^{(-1)}- \tfrac{1}{\lambda_1} (Ax^{(-1)}-b)$, which ensures that $v_1^Tg\kz = 0$, i.e.,  the first component in the sum \eqref{gdynamicsnocycling} corresponding to the direction $v_1$ is eliminated.
Now, suppose that $\alpha = 2/\lambda_1$. Then, for all $\lambda_j>0$, $|1-\tfrac{2\lambda_j}{\lambda_1}|<1$ so that $|1-\tfrac{2\lambda_j}{\lambda_1}|^k \to 0$ as $k\to \infty$. Therefore, GD with the initialization above, and the fixed step size $\alpha = 2/\lambda_1$ is guaranteed to converge whenever $A$ is positive definite. Moreover, because a fixed step length is used, the connection with the PM still holds, and the gradient converges in the direction of $v_n$. This is summarized in the following theorem.
\begin{theorem}
Gradient descent with the fixed step size $\alpha = 2/\lambda_1$ is guaranteed to converge when applied to the problem \eqref{probmin}, when $A$ is positive definite and $x\kz = x^{(-1)}- \tfrac{1}{\lambda_1} (Ax^{(-1)}-b)$ for some $x^{(-1)}\in \R^n$.
\end{theorem}

On the other hand, for all $\lambda_j<0$, $|1-\tfrac{2\lambda_j}{\lambda_1}|>1$ so these coefficients grow as $k\to \infty$. Moreover $|1-\tfrac{2\lambda_j}{\lambda_1}|>|1-\tfrac{\lambda_j}{\lambda_1}|$, so the more aggressive step size $\alpha = 2/\lambda_1$ leads to faster escape from the saddle point.

We remark that the step length $2/(\lambda_1 + \lambda_n)$ is known to be `optimal' for (non-accelerated) GD when $A$ is positive definite. However, this step size relies upon knowledge of $\lambda_n$, which is often unknown when the algorithm is initialized. The strategy proposed above does not require knowledge of $\lambda_n$. Note also that $2/\lambda_1$ is a larger step than $2/(\lambda_1 + \lambda_n)$.

\subsection{GD-EIG with a Kick}
\label{Section_GDKick}

An eigen-enriched GD algorithm called GD-EIG-Kick is presented now. The `base' algorithm is simply GD(-EIG) with a fixed step size, but every $s+1$ iterations, the algorithm attempts to take a longer step (a `kick'), where the step size is related to $\lambda_n\kk$, the current approximation to the leftmost eigenvalue. GD-EIG-Kick is presented as Algorithm~\ref{alg:GDKick}.

\begin{algorithm}
	\begin{algorithmic}[1]
        \STATE {\bf Input:}  Initial point $x^{(0,0)}$, $x^{(0,-1)}=x^{(0,0)}$, $\alpha\in (0,1/L]$, $\beta\in [0,1]$, iteration counter $s$.
		\FOR {$k = 0,1,2,\dots$}
		\FOR {$\ell = 1,2,\dots,s$}
        \STATE $x^{(k,\ell)} = x^{(k,\ell-1)} - \alpha g(x^{(k,\ell-1)} ) + \beta(x^{(k,\ell-1)}-x^{(k,\ell-2)})$\label{GDupdate}
        \ENDFOR
        \STATE Compute $g(x^{(k,s)})$
        \STATE $\nu_1^{(k)} = (g(x^{(k,s-1)})^Tg(x^{(k,s)}))/(g(x^{(k,s-1)})^Tg(x^{(k,s-1)}))$
        \STATE $\lambda_n^{(k)} = (1+ \beta -\nu_1^{(s)})/\alpha$
        \IF {$\lambda_n^{(k)} \neq 0$}\label{Step:safeguardsingular}
        \STATE $\tilde x = x^{(k,s)} - |1/\lambda_n^{(k)}| g(x^{(k,s)})$ \label{Step:safeguardnegative}
        \STATE $\hat x = x^{(k,s)} - \alpha g(x^{(k,s)})$ 
        \IF {$f(\tilde x)< f(\hat x)$}
        \STATE $ x^{(k+1,0)} = \tilde x$ 
        \ELSE
        \STATE $ x^{(k+1,0)} = \hat x$
        \ENDIF
        \ELSE
        \STATE $ x^{(k+1,0)} = x^{(k,s)} - \alpha g(x^{(k,s)})$
        \ENDIF
        \STATE $ x^{(k+1,-1)} = x^{(k+1,0)}$
		\ENDFOR
	\end{algorithmic}
	\caption{GD-EIG with a Kick}
	\label{alg:GDKick}
\end{algorithm}

The motivation for GD-EIG-Kick is as follows. Recall \eqref{gdynamicsnocycling} and the discussion in Section~\ref{Section:GDdynamics}. If $A$ is PD then $\lambda_1\geq \dots\geq \lambda_n>0$, and components of the gradient corresponding to larger eigenvalues go to zero faster than components corresponding to smaller eigenvalues. Running $s$ iterations of GD-EIG with $\alpha = 1/\lambda_1$ will eliminate the $v_1$ component, and will shrink the larger components, but if $\lambda_1\gg \lambda_n$, then the smaller components will remain almost unchanged. Using the Rayleigh quotient to give the approximation $\lambda_n\kk$, ensures that $1/\lambda_n\kk \geq 1/\lambda_1$, and so taking this larger step will hopefully be beneficial for shrinking the components of the gradient corresponding to smaller eigenvalues. However, this larger step can also increase components of the gradient corresponding to larger eigenvalues, so the step should only be accepted if it leads to a larger reduction in the function value than the fixed step $\alpha$, (but the norm of the gradient is allowed to increase). This `fixed step then kick' process then repeats. 

On the other hand, if $A$ is indefinite, then $\lambda_n\kk$ will still approximate the leftmost eigenvalue, but now the approximation could be negative. In Step~\ref{Step:safeguardnegative}, it is suggested to try to take the absolute value of the reciprocal of $\lambda_n\kk$ as the step size. As discussed previously, this will hopefully lead to faster evasion of the saddle point. Note that, if the goal was to locate a stationary point (for example, if one wished to solve $Ax=b$, rather than \eqref{probmin}), then dropping the absolute values and taking the step size $1/\lambda_n\kk$ gives a kick toward the saddle.

Two safeguards are included in GD-EIG-Kick. Firstly, if $A$ is singular, then it is possible that the estimate of the leftmost eigenvalue is zero, so Step~\ref{Step:safeguardsingular} is included to ensure that a `kick' step is only considered if $\lambda_n\kk \neq 0$. Secondly, assuming that $\lambda_n\kk \neq 0$, a longer `kick' step will only be accepted if it leads to a greater reduction in the function value than a corresponding fixed step size. Hence, the convergence and rate of GD-EIG-Kick is the same as for GD with a fixed step size.

GD-EIG-Kick makes use of the `freely available' eigen-information, so it is no more expensive than vanilla GD; the dominant cost is one matrix vector product per iteration. Moreover, the connection with the PM still holds, because we can view GD-EIG-Kick, simply as repeatedly restarted GD/GD-EIG. Thus, GD-EIG-Kick will still provide and approximation to $\lambda_n$, and the gradient still aligns in the direction $v_n$ as $k\to \infty$.

\begin{remark}
Although not explicitly written, it is allowed to use the fixed step $\alpha = 2/\lambda_1$ in Step~\ref{GDupdate} of GD-EIG-Kick (Algorithm~\ref{alg:GDKick}). However, in this case, one \emph{must} take a step of size $\alpha = 1/\lambda_1$ whenever $\ell = 1$ (i.e., every time the inner loop begins), and then for $2\leq \ell \leq s$, the step size $\alpha = 2/\lambda_1$ is allowed. Failure to take the first loop step of size $\alpha = 1/\lambda_1$ will cause the algorithm to diverge.
\end{remark}

\section{Numerical Experiments}
\label{Section_Numerical}

In this section, several experiments are presented to numerically verify the results described in this work. In particular, experiments are presented showing that (1) an approximation to the leftmost eigen-pair can be recovered in practice; (2) that the fixed step length $\alpha = 2/\lambda_1$, combined with the appropriate initialization, leads to favourable practical performance compared with the fixed step length $\alpha = 1/\lambda_1$; and (3) that GD-EIG-Kick has improved practical behaviour compared with vanilla GD with a fixed step size.\footnote{The codes to reproduce numerical experiments are available at \url{https://github.com/Optimization-and-Machine-Learning-Lab/gradient_descent_and_power_method.git}}

\subsection{Estimating the left eigen-pair}

The purpose of the first set of numerical experiments is to confirm numerically some of the main observations from the paper; specifically, that an estimate of the leftmost eigen-pair can be obtained in practice using GD-EIG and GD-EIG-Kick. 

\subsubsection{Approximation using GD-EIG}

\begin{figure}
    \centering

    \includegraphics[width=0.3\textwidth]{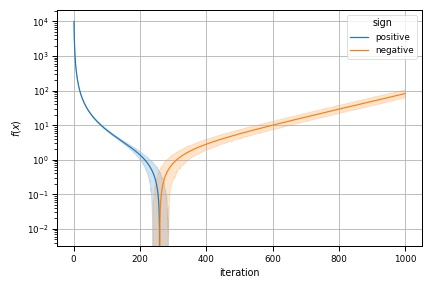}
    \includegraphics[width=0.3\textwidth]{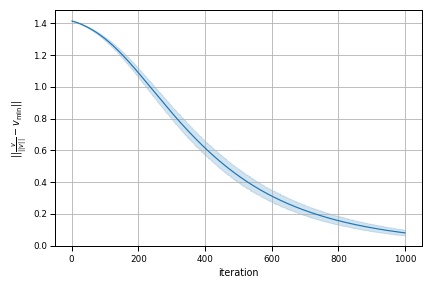}
    \includegraphics[width=0.3\textwidth]{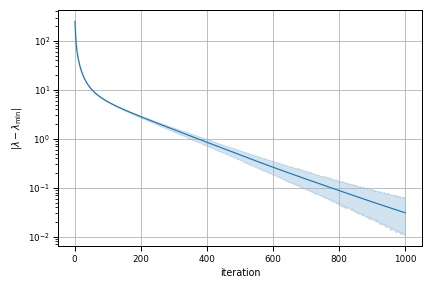}

    \caption{Figure showing the results of the first numerical experiment. The left plot shows the evolution of the function values, the middle plot shows the accuracy of the eigenvector approximation, and the right plot shows the accuracy of the eigenvalue estimate.}
    \label{fig:remark3}
\end{figure}
To this end, 100 random symmetric indefinite matrices $A \in \R^{200 \times 200}$ were generated, and in all cases $b = 0$, so that $f(x) = \tfrac12 x^TAx$. For each matrix, a random initial starting point $x\kz\in \R^{200}$ was generated, and GD-EIG, with a step size of $\alpha = 1/\lambda_1$, was run for 1,000 iterations.
Note that, because the matrix $A$ is indefinite, there is a saddle point, and
the optimization problem~\eqref{probmin} 
is unbounded from below.
Figure~\ref{fig:remark3} shows the results of this experiment.

The left plot in Figure~\ref{fig:remark3} shows the evolution of $f(x^{(k)})$ as the algorithm progresses. Note that, because a log-scale is used, positive and negative function values are differentiated by colour (blue corresponds to $f(x^{(k)})\geq 0$, while red corresponds to $f(x^{(k)}) < 0$). All 100 runs/problem instances are plotted in this figure, giving a shaded region, with the average over the 100 runs given by the darker line. This confirms that GD-EIG is solving the problem~\eqref{probmin}, as expected ($f(x\kk) \to -\infty$ as $k\to \infty$). Notice that the function value decreases rapidly initially, but then slows down as it nears the saddle point. After just over 200 iterations on average, a direction of negative curvature is found, the function value becomes negative and the iterates slowly escape the saddle point.

In the middle plot we show the accuracy of the left-most eigenvector estimate, and the right plot shows the left-most eigenvalue estimate. Recall that the Rayleigh quotient provides a quadratically accurate approximation to the eigenvalue, while the approximation to the eigenvector is only linear, which corresponds with what is shown numerically (i.e., it takes more iterations to generate an accurate approximation to the leftmost eigenvector, compared with the corresponding eigenvalue). This confirms that an estimate of the leftmost eigen-pair is recovered numerically. Recall \eqref{pmrate}, which gives the approximation rate for the leftmost eigenvalue, which is the same when using GD-EIG with a fixed step size, or using the PM. That is, we do not claim a faster rate of convergence using GD-EIG, but we do show that an approximation to the leftmost eigen-pair is available.

\subsubsection{Approximation using GD-EIG-Kick}

\begin{figure}
    \centering
    \includegraphics[width=0.3\textwidth]{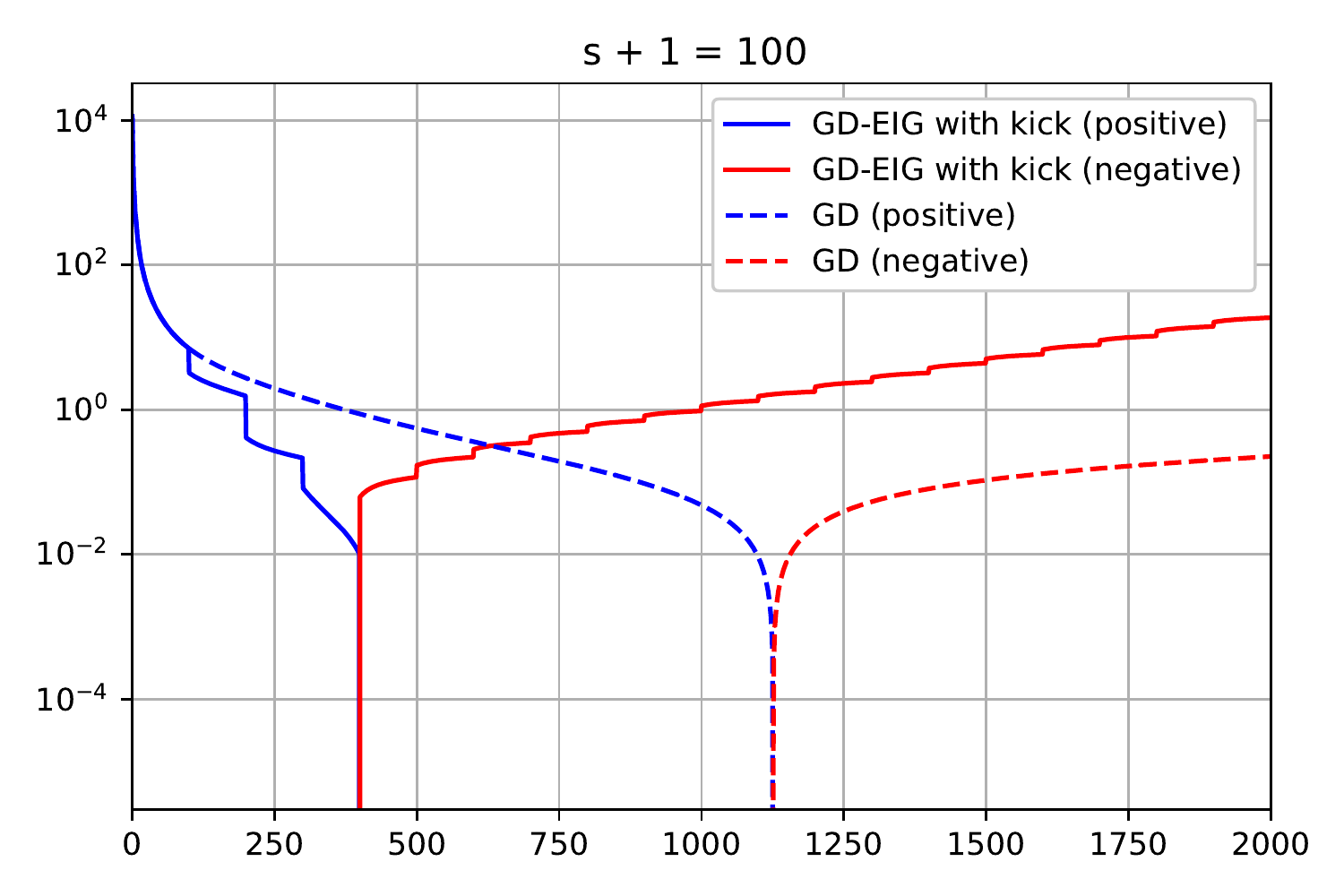}
    \includegraphics[width=0.3\textwidth]{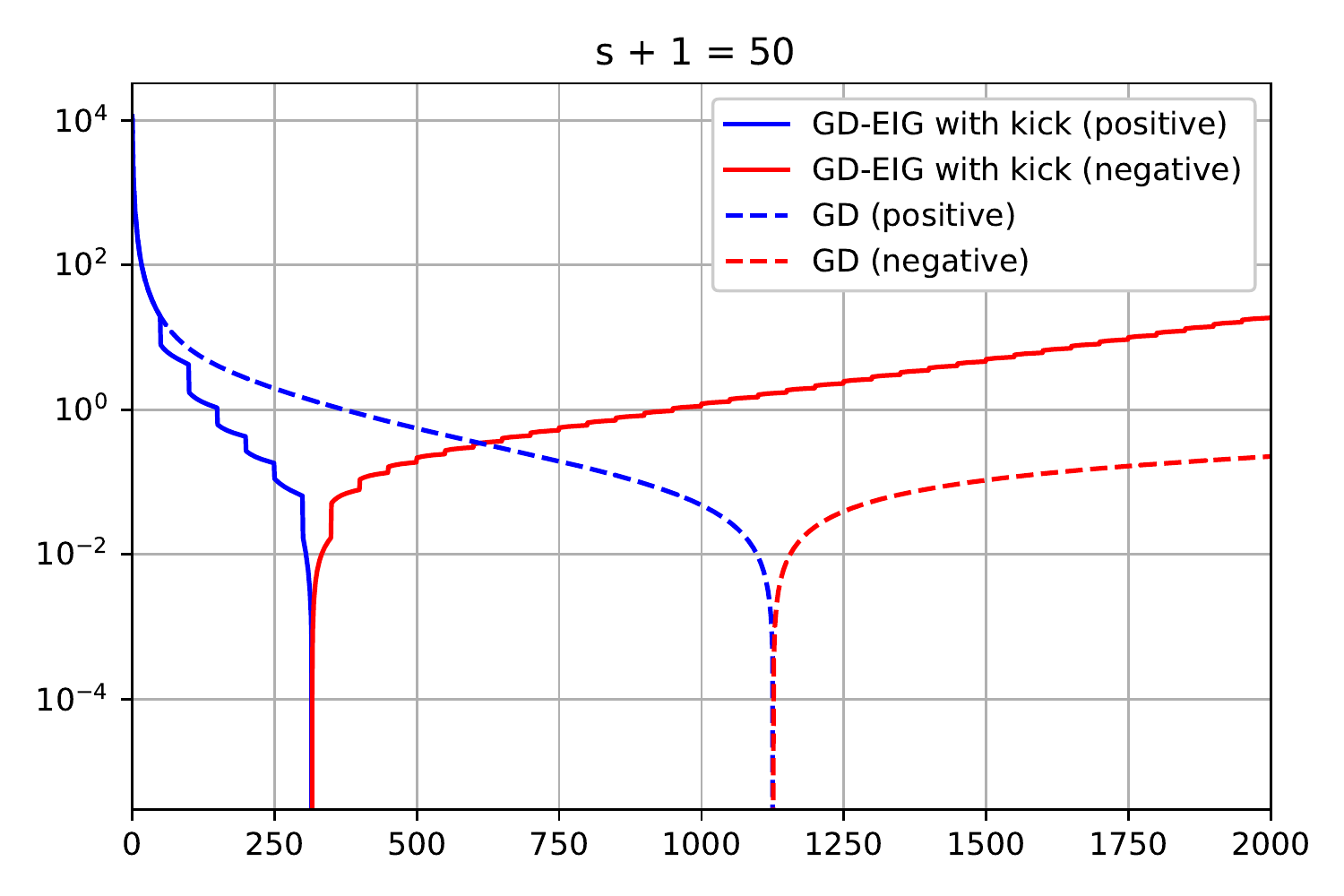}
    \includegraphics[width=0.3\textwidth]{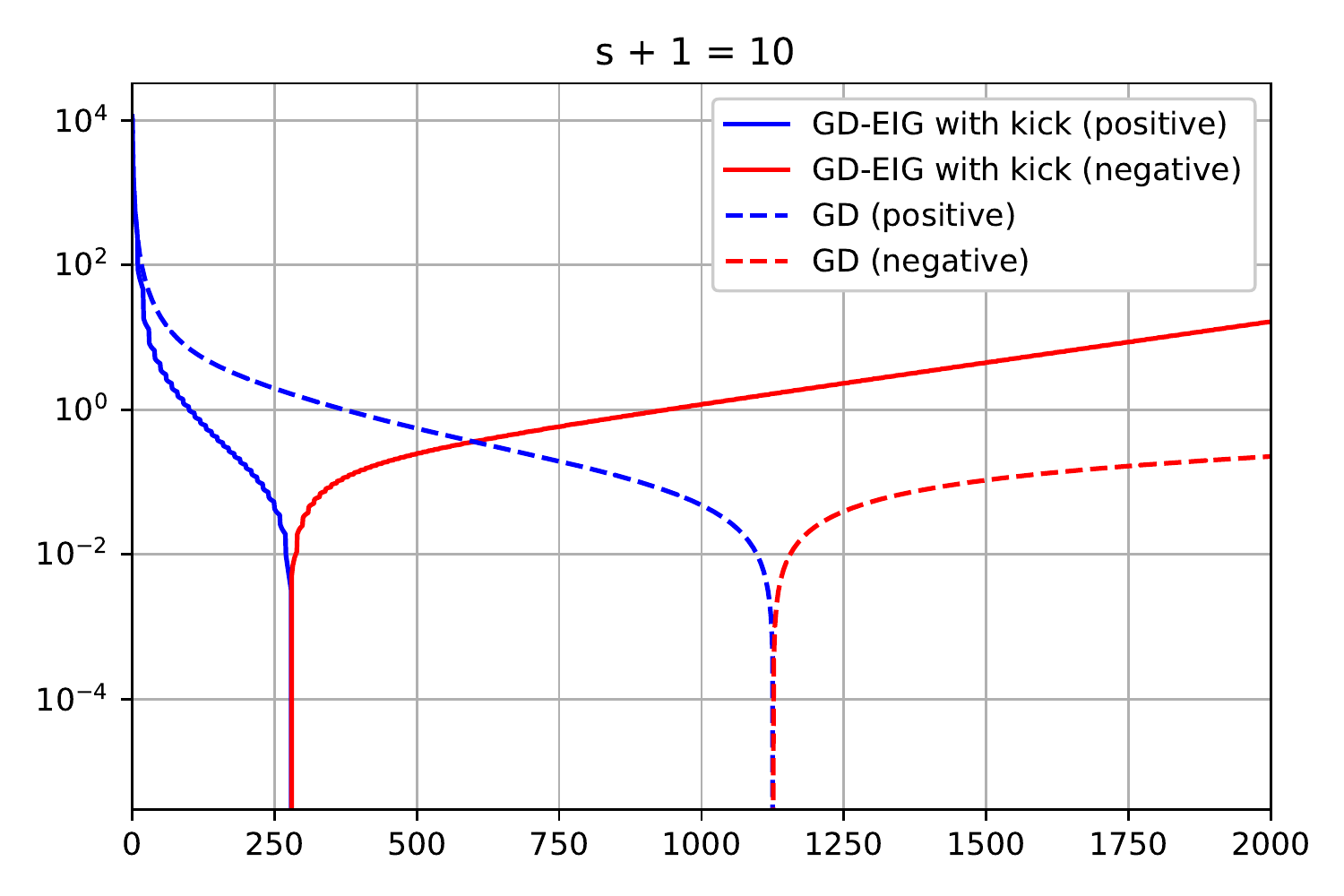}
    
    \includegraphics[width=0.3\textwidth]{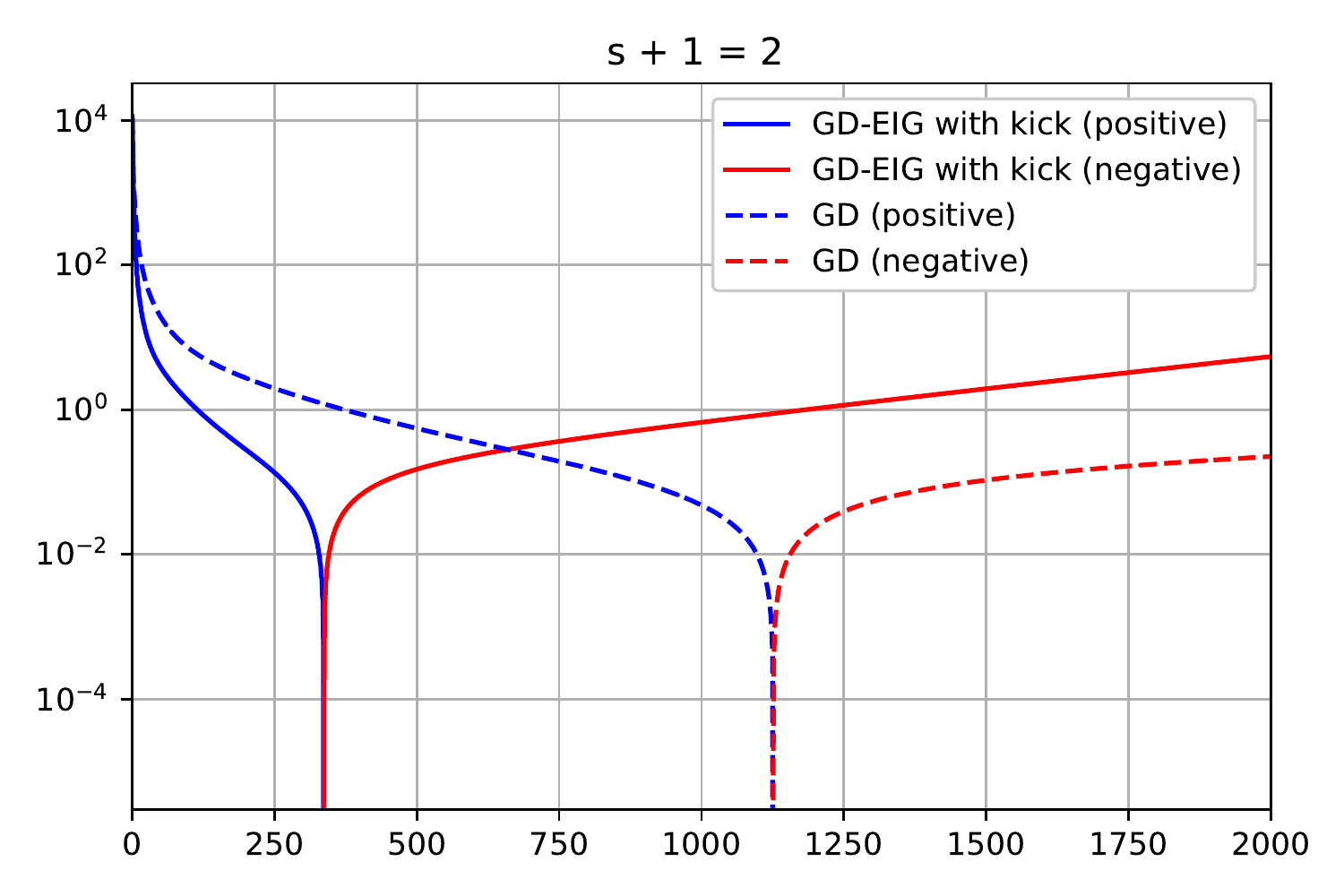}
    \includegraphics[width=0.3\textwidth]{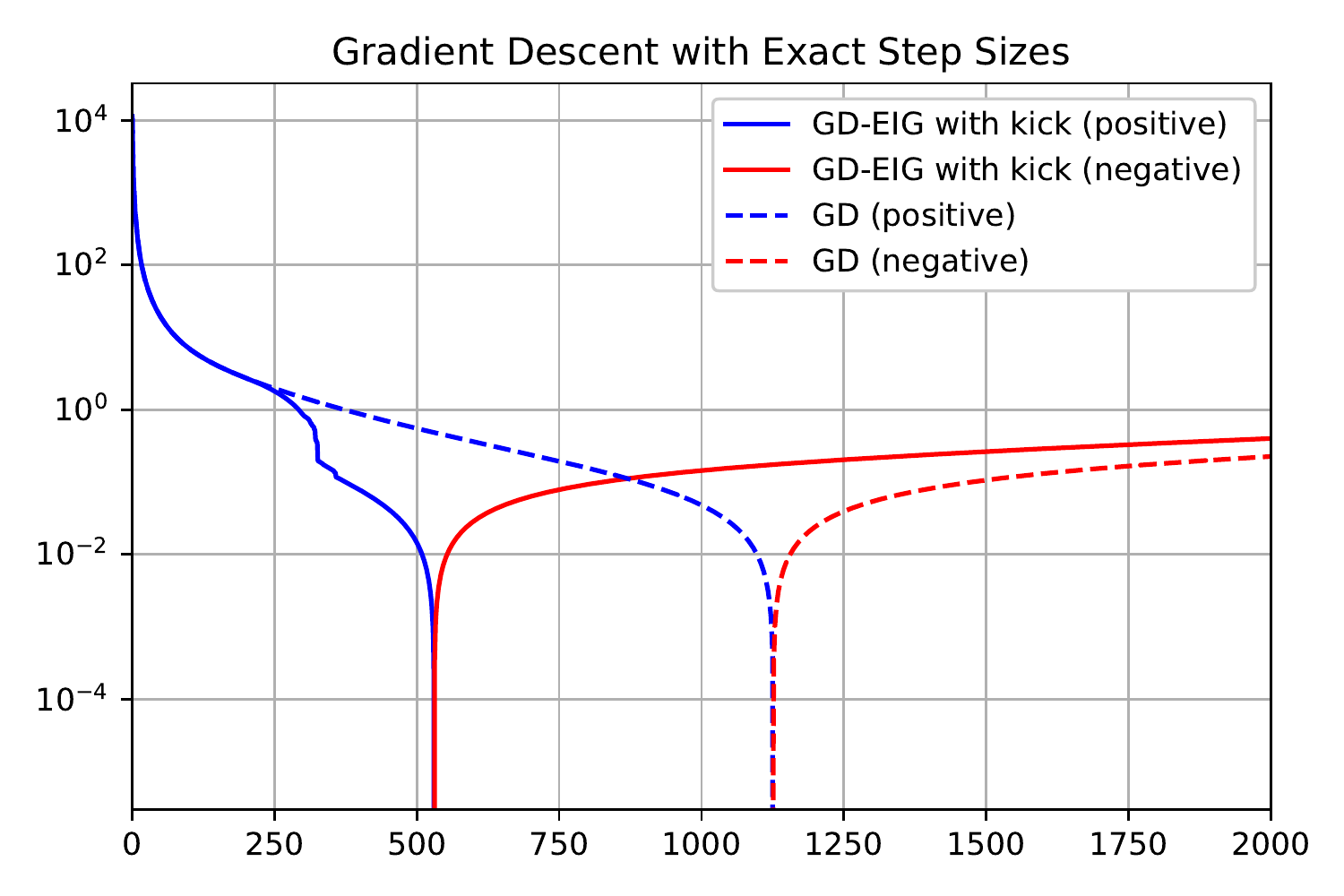}

    \caption{GD-EIG with kick.}
    \label{fig:gd-eig_withKick}
\end{figure}
Another experiment is performed to verify that GD-EIG-Kick can also be used to estimate a left eigen-pair. The experimental set up is similar to the previous one, with randomly generated $A\in \R^{200\times 200}$, $b = 0$, and $x\kz \in \R^{200}$. GD-EIG-Kick, with a fixed step size $\alpha = 1/\lambda_1$, was run for 2000 iterations, and the results are shown in Figure~\ref{fig:gd-eig_withKick}.

Figure~\ref{fig:gd-eig_withKick} shows the behaviour of GD-EIG-Kick with various choices of $s$ on the problem described above. Again, this plot distinguishes between positive and negative function values by colour: blue corresponds to $f(x\kk)\geq 0$ and red corresponds to $f(x\kk)< 0$. In each plot, the dashed line corresponds to GD with a fixed step size of $\alpha = 1/\lambda_1$ (this is fixed in all plots). For this (representative) problem instance using GD, the function value decreases rapidly initially, but then the iterates become trapped by the saddle point and the function values decrease slowly, until at approximately 1100 iterations, a direction of negative curvature is found, the function value switches to negative, and the iterates slowly escape the saddle point. 

Also shown in each plot is the behaviour of GD-EIG-Kick (solid line) with the stated $s$ value. In all cases, GD-EIG-Kick finds a direction of negative curvature more quickly than GD, and thus, escape from the saddle point is much more rapid (the slope of the red (solid) curve is slightly steeper than for GD (dashed line) in each case). This confirms that using GD-EIG-Kick is beneficial in practice. It is also clear that the choice of $s$, the number of iterations before a kick, impacts the practical performance of the algorithm. For this experiment, the best performance occurred when $s+1=10$ (9 iterations with a fixed step size, and then one kick iteration). It is also apparent that the `kick' iterations lead to a clearly visible reduction in the function values (the `staircase' like pattern in the function values).  In the plot, we also record the behaviour of Gradient/Steepest Descent with an exact step size as a benchmark ($s = 0$ i.e., always `kick'). Notice that using an exact step size is better than GD with a fixed step size, but GD-EIG-Kick is still better in all cases.

\subsection{Investigation of step sizes in GD}

The purpose of the experiment in this section is to investigate the behaviour of GD-EIG with the longer step size $\alpha = 2/\lambda_1$. A symmetric positive definite matrix $A\in \R^{1000 \times 1000}$ was generated, as well as the optimal solution $\xs \in \R^{1000}$, with $b$ computed as $b = A\xs$. A point $x^{(-1)}\in\R^{1000}$ was generated, and the starting point for the experiments was set as $x\kz = x^{(-1)}- (1/\lambda_1)(Ax^{(-1)}-b)$ (i.e., this is the `smart' initialization strategy, which ensures that the component of the gradient in the direction $v_1$ is eliminated). Figure \ref{fig:steplength} shows the results. Note that, while 100 problem instances of the type described above were generated, the results of a single \emph{representative} problem instance are reported.

The lines correspond to the step sizes $\alpha = 1/\lambda_1$ (blue), $\alpha = 2/\lambda_1$ (red), and $\alpha = 2/(\lambda_1+\lambda_n)$ (green). The solid lines correspond to $\beta = 0$ (no momentum), the dashed lines correspond to $\beta = 0.5$, and the dotted lines correspond to $\beta = 0.8$. Notice that, regardless of the choice of momentum parameter $\beta$, the red and green lines coincide, showing no discernible difference, and both choices led to improvements over the step size $\alpha = 1/\lambda_1$. This is important because, while $\lambda_1$ is often available when the algorithm begins, $\lambda_n$ is usually unknown, so the step size $\alpha = 2/\lambda_1$ is convenient in practice. These experiments also show that momentum can be helpful in terms of the practical behaviour of the method, with these experiments suggesting larger values of the momentum parameter correspond to a reduction in the number of iterations.

\begin{figure}[ht]
    \centering
    \includegraphics[width=0.3\textwidth]{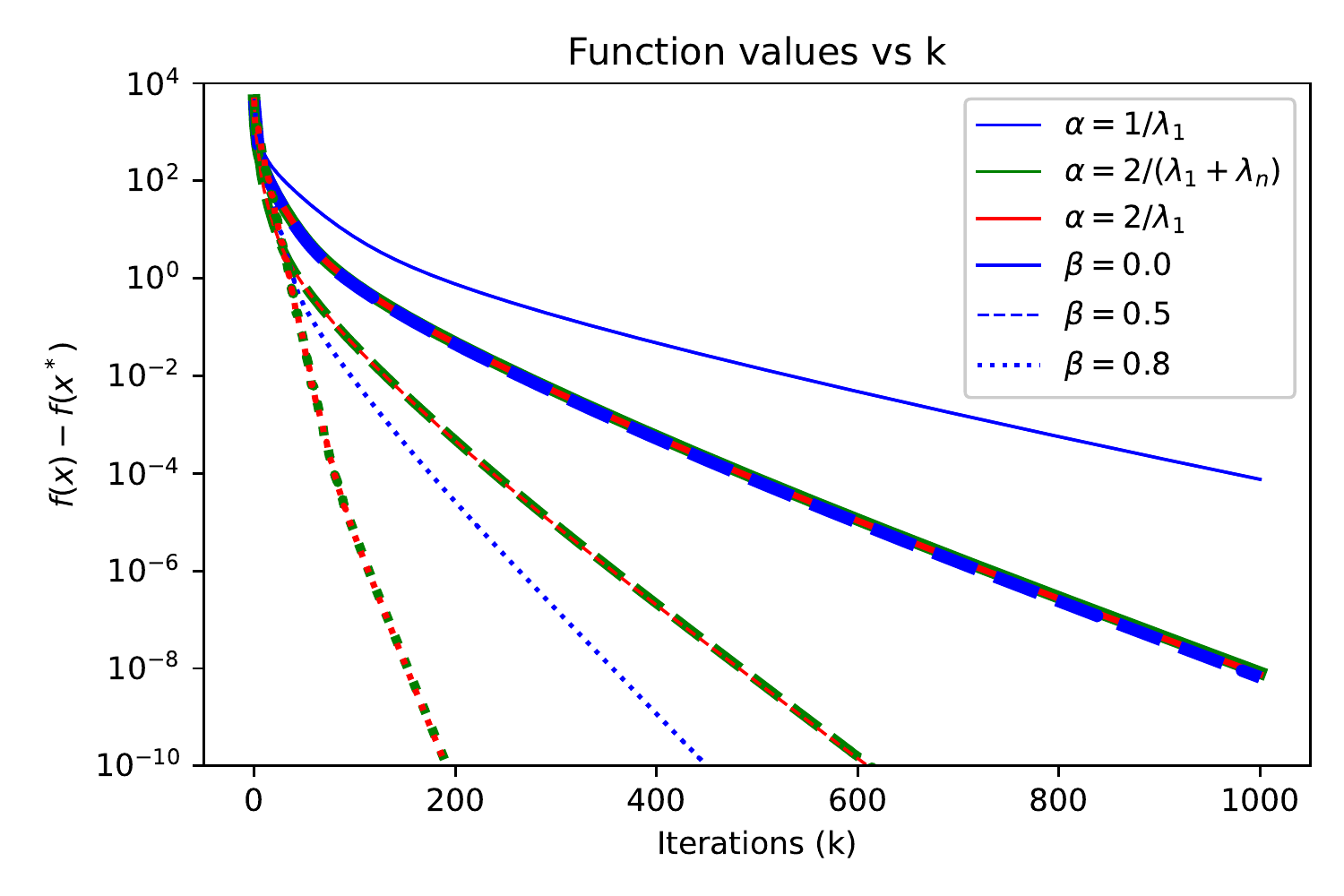}
    \includegraphics[width=0.3\textwidth]{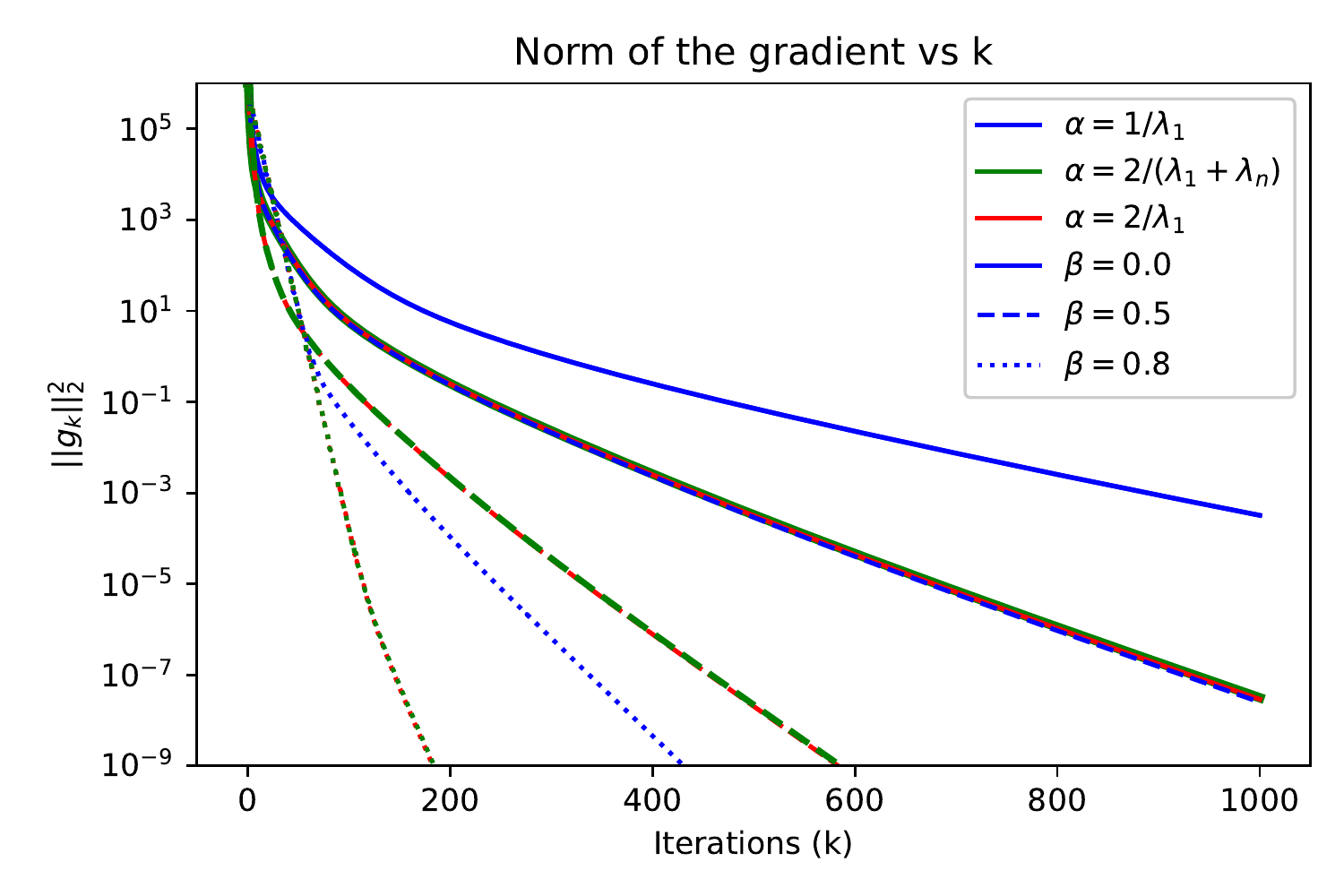}    
    \includegraphics[width=0.3\textwidth]{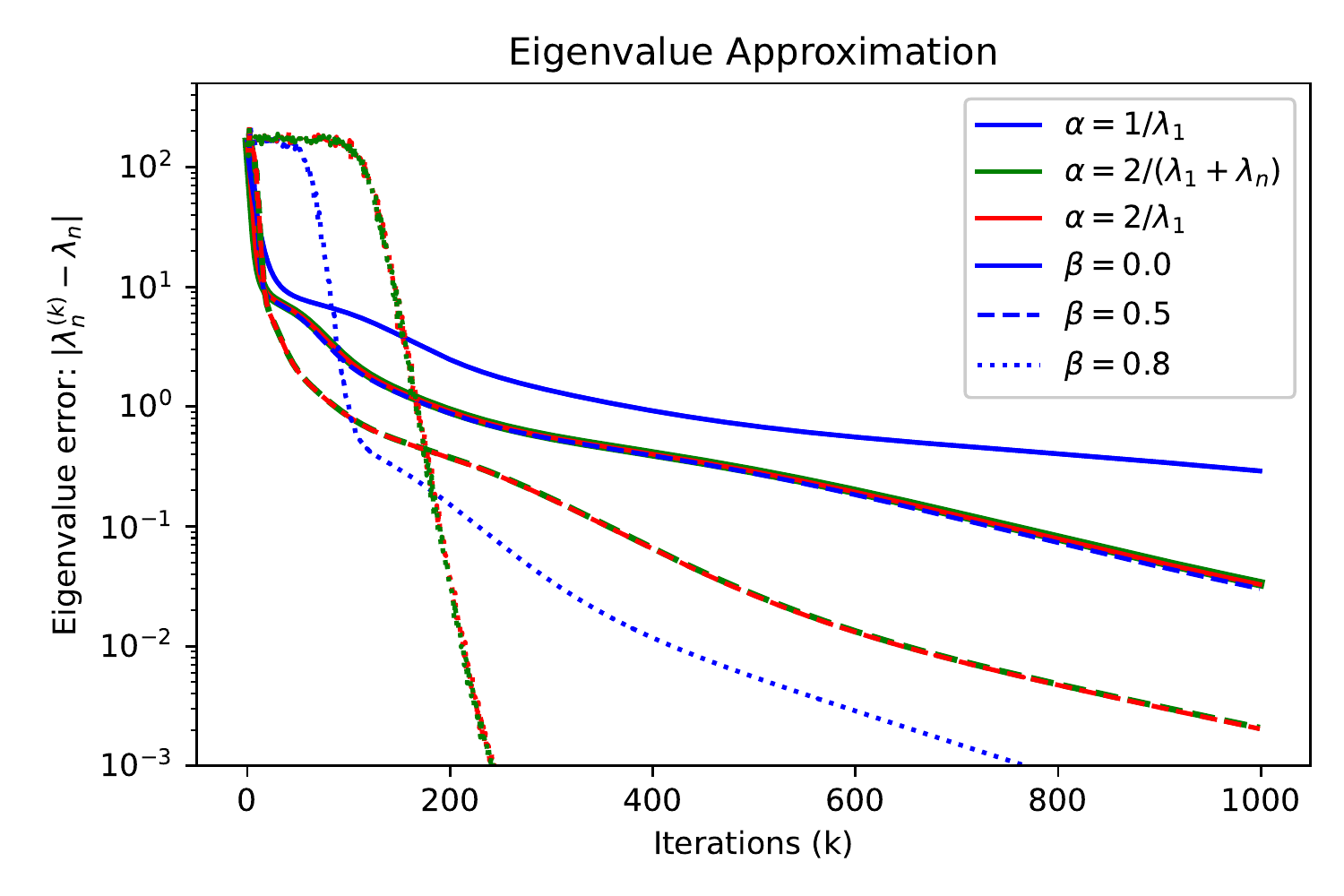}    
        
    \caption{Investigating the choice of step length in Gradient Descent.}
    \label{fig:steplength}
\end{figure}

\subsection{Investigating GD-EIG with a Kick}

In this section the performance of GD-EIG-Kick is studied. A symmetric positive definite matrix $A\in \R^{1000 \times 1000}$ was generated, as well as the optimal solution $\xs \in \R^{1000}$, with $b$ computed as $b = A\xs$. For all experiments, the step size $\alpha\kk = 2/\lambda_1$ was used, coupled with the appropriate initial point. The experiment compares GD with a fixed step size, with GD-EIG-Kick with 2 different inner loop lengths (i.e., 2 choices of inner iteration counter $s$), and the algorithms were also compared with the Accelerated Gradient Method as a benchmark. (Recall that the Accelerated Gradient Method requires knowledge of $\lambda_n>0$ in advance, which is not assumed for GD-EIG-Kick.) Figure \ref{fig:gdeigkick} shows the results of this experiment.  

The top left plot shows the evolution of the function value error $f(x\kk) - f^*$, the top middle plot shows the step size for GD-EIG-Kick, for both $s+1=20$ and $s+1 = 100$, and the top right plot shows the evolution of the gradient norm $\|g\kk\|_2^2$. Consider the top left plot. The blue line corresponds to 999 iterations of GD with the fixed step size $\alpha\kk = 2/\lambda_1$, coupled with the appropriate initial point. For the 1000th iteration, a `kick' step (using the Rayleigh quotient) was used, which corresponds to the rapid decrease in the function value error in the last iteration. GD-EIG-Kick (the red and black lines) clearly outperform vanilla GD, and the choice $s+1=20$ (19 iterations with a fixed step size, followed by a kick) is slightly better than $s+1 = 100$, suggesting that it is advantageous to take a `kick' step more often. Note that the accelerated gradient method (the green line) performs the best, but it is not drastically better than GD-EIG-Kick with $s+1 = 20$, and it also requires prior knowledge of $\lambda_n$.

Notice that for each inner loop, the first step is $\alpha^{(k,1)} = 1/\lambda_1$, the next $s-1$ iterations correspond to a step of size $\alpha^{(k,j)} = 1/\lambda_1$ for $2\leq j \leq s$, and the last $s+1$th step is $\alpha^{(k,s)} = 1/|\lambda_n\kk|$, where $\lambda_n\kk$ is the approximation to $\lambda_n$ generated by the algorithm. It is important to note that the step size $\alpha^{(k,1)} = 1/\lambda_1$ is necessary at the start of every inner loop; without it, the longer step size $2/\lambda_1$ will cause the algorithm to diverge. The top middle plot shows this pattern, and also shows that every $s+1$th step size is much larger, so the `kick' helps the algorithm to make faster progress toward the solution.

The top right plot shows the evolution of the gradient norm. For GD with the fixed step size, the curve is smooth, and decreases slowly. For GD-EIG-Kick, with either choice of $s$, the norm of the gradient decreases more rapidly. Notice that at every $s+1$th iteration, the norm of the gradient increases, corresponding to the `kick'. These kick steps also correspond to the larger decrease in the function value --- a staircase pattern --- as previously mentioned.

The experiment described above was repeated for a matrix $A$ that is only positive semi-definite (it had 50 eigenvalues that were set to zero). Similar behaviour was observed for this problem. Note that the accelerated gradient method requires the objective function to be strongly convex, so it was not employed here.
\begin{figure}[ht]
    \centering
    \includegraphics[width=0.3\textwidth]{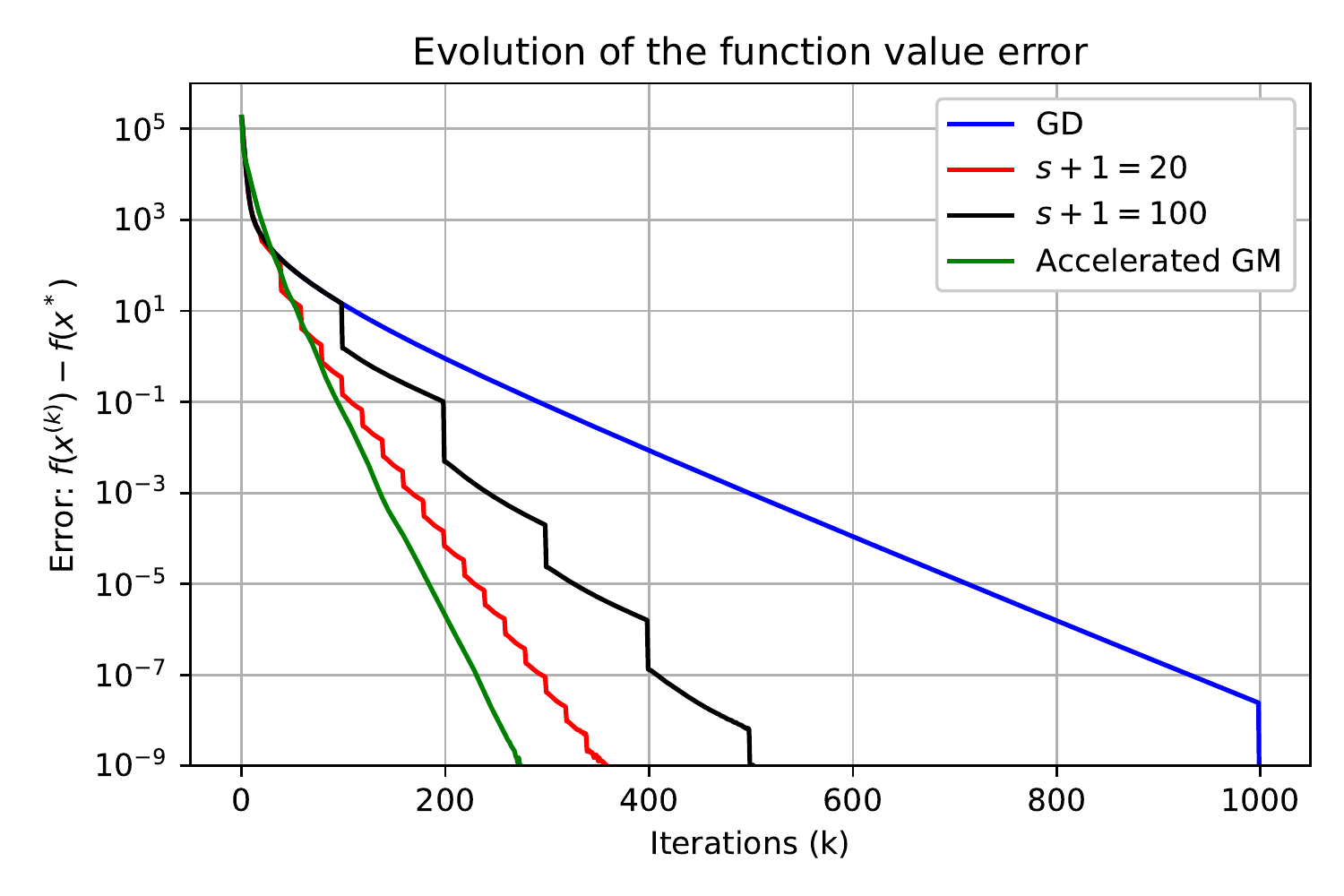}
    \includegraphics[width=0.3\textwidth]{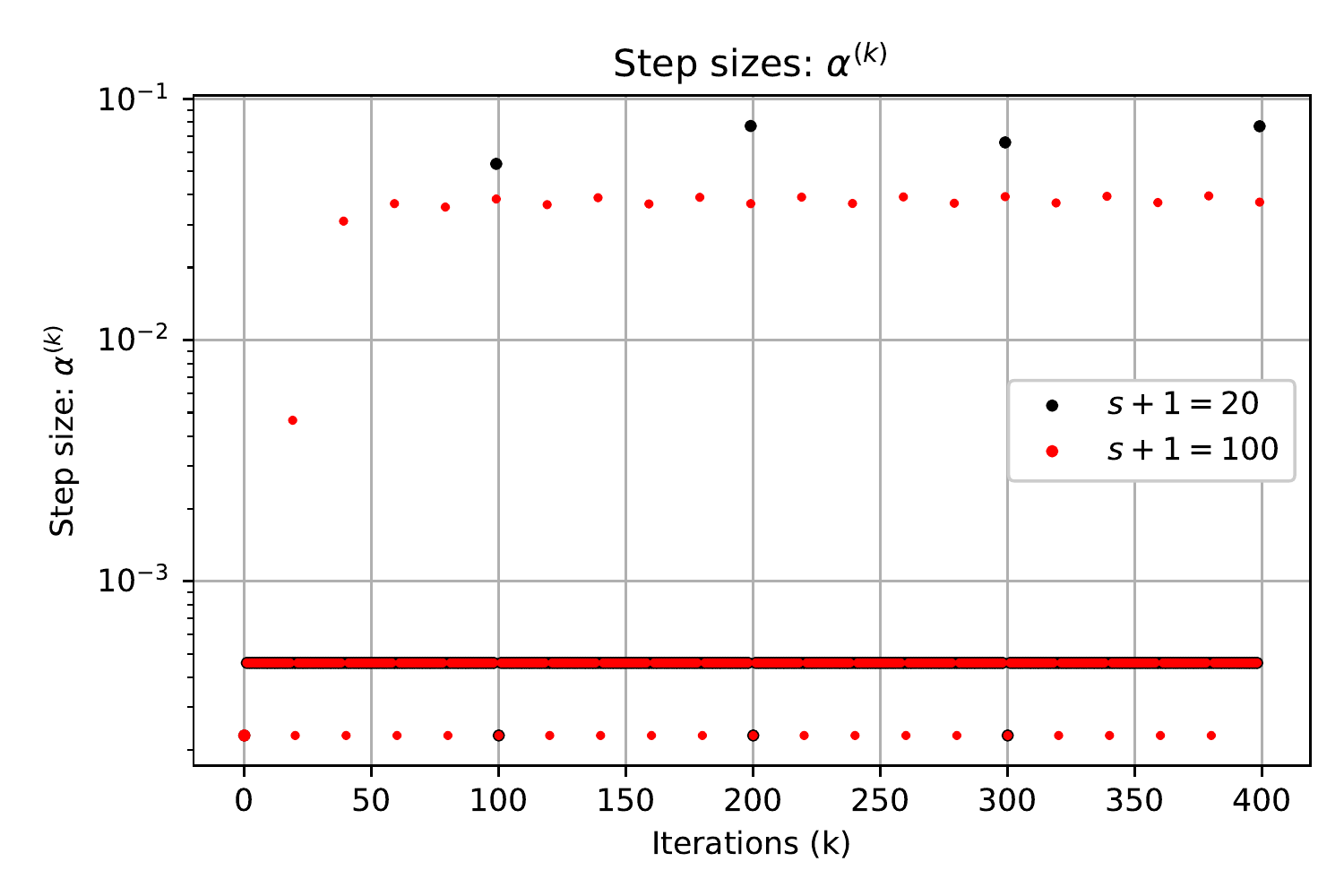}
    \includegraphics[width=0.3\textwidth]{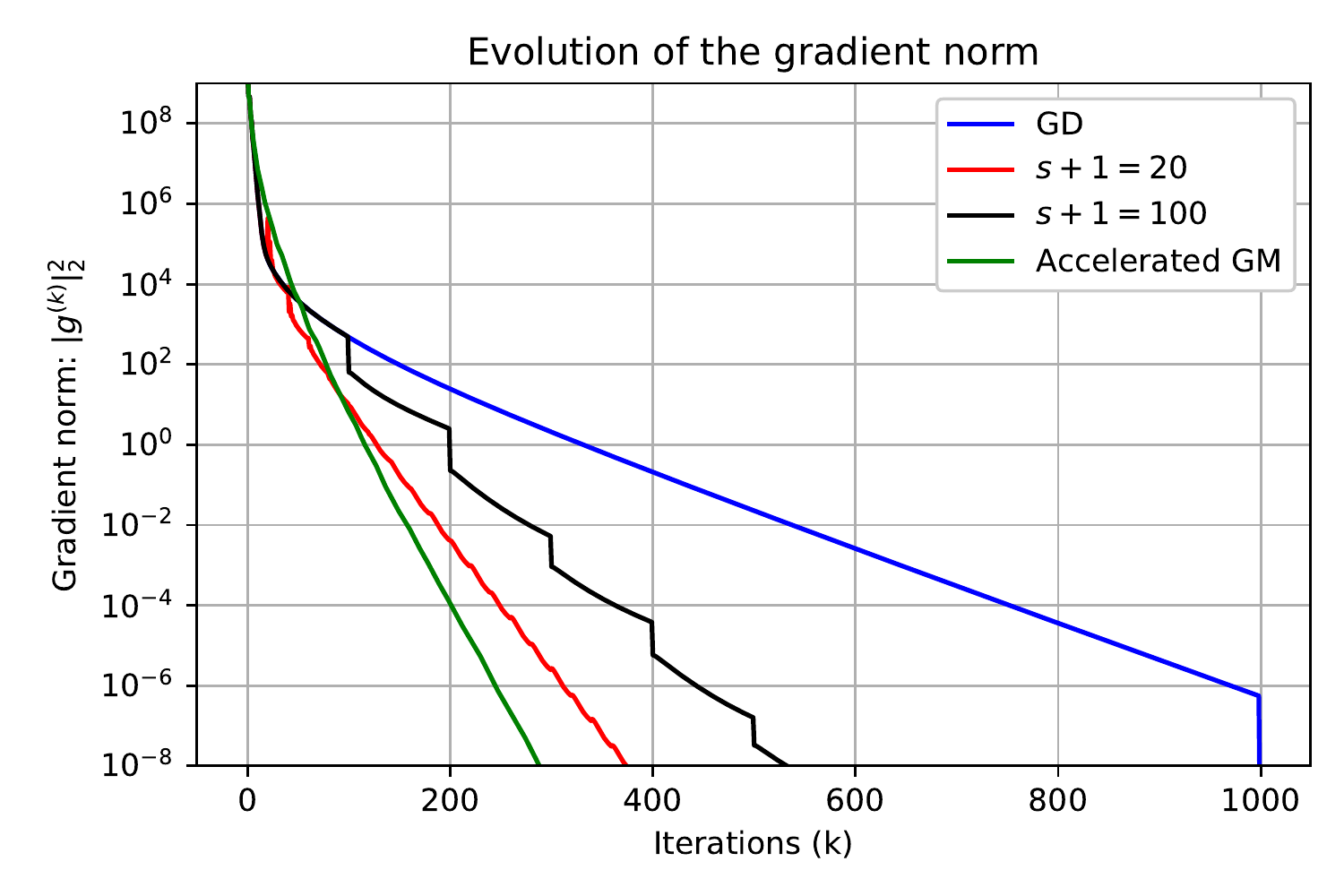}
    
    \includegraphics[width=0.3\textwidth]{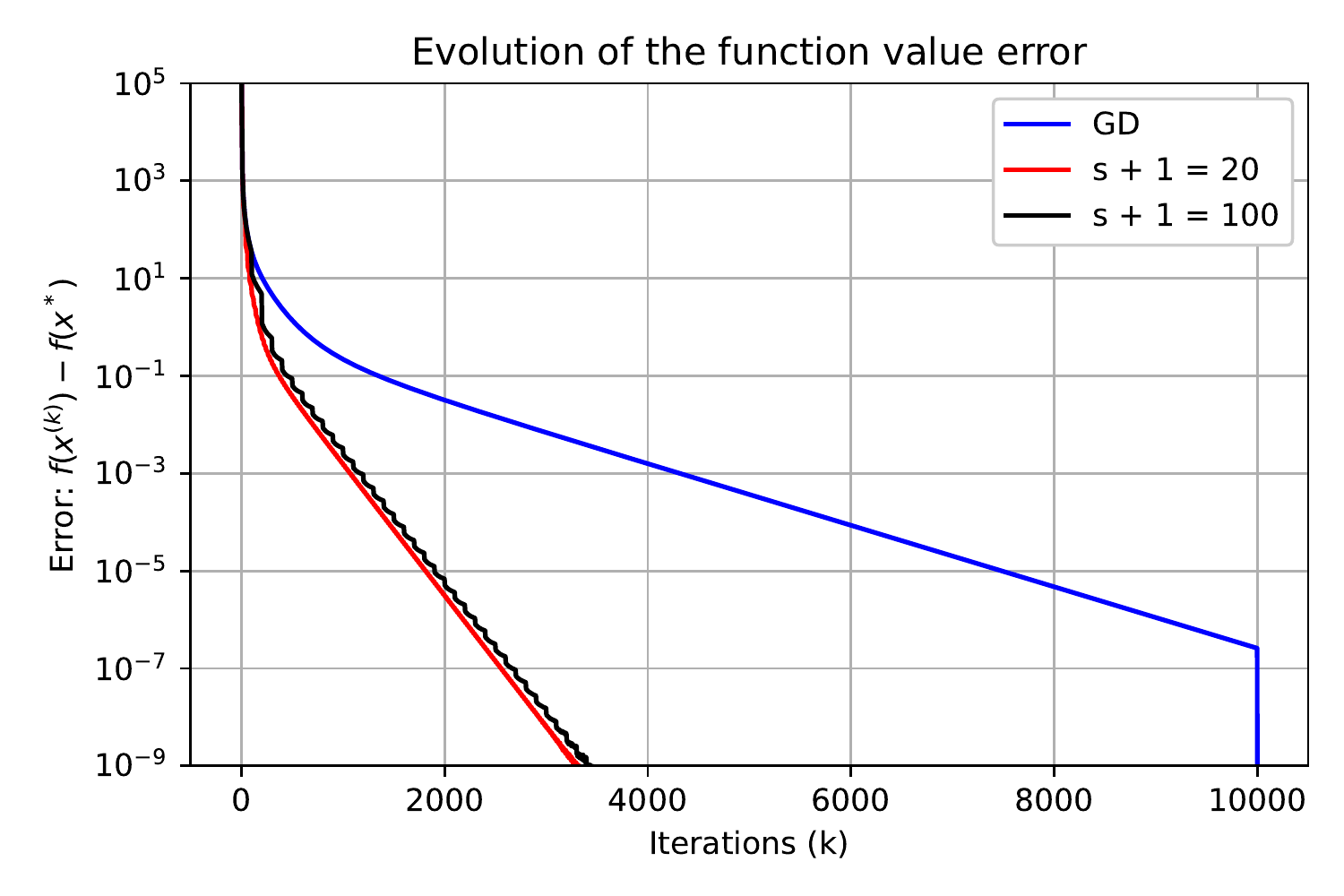}
    \includegraphics[width=0.3\textwidth]{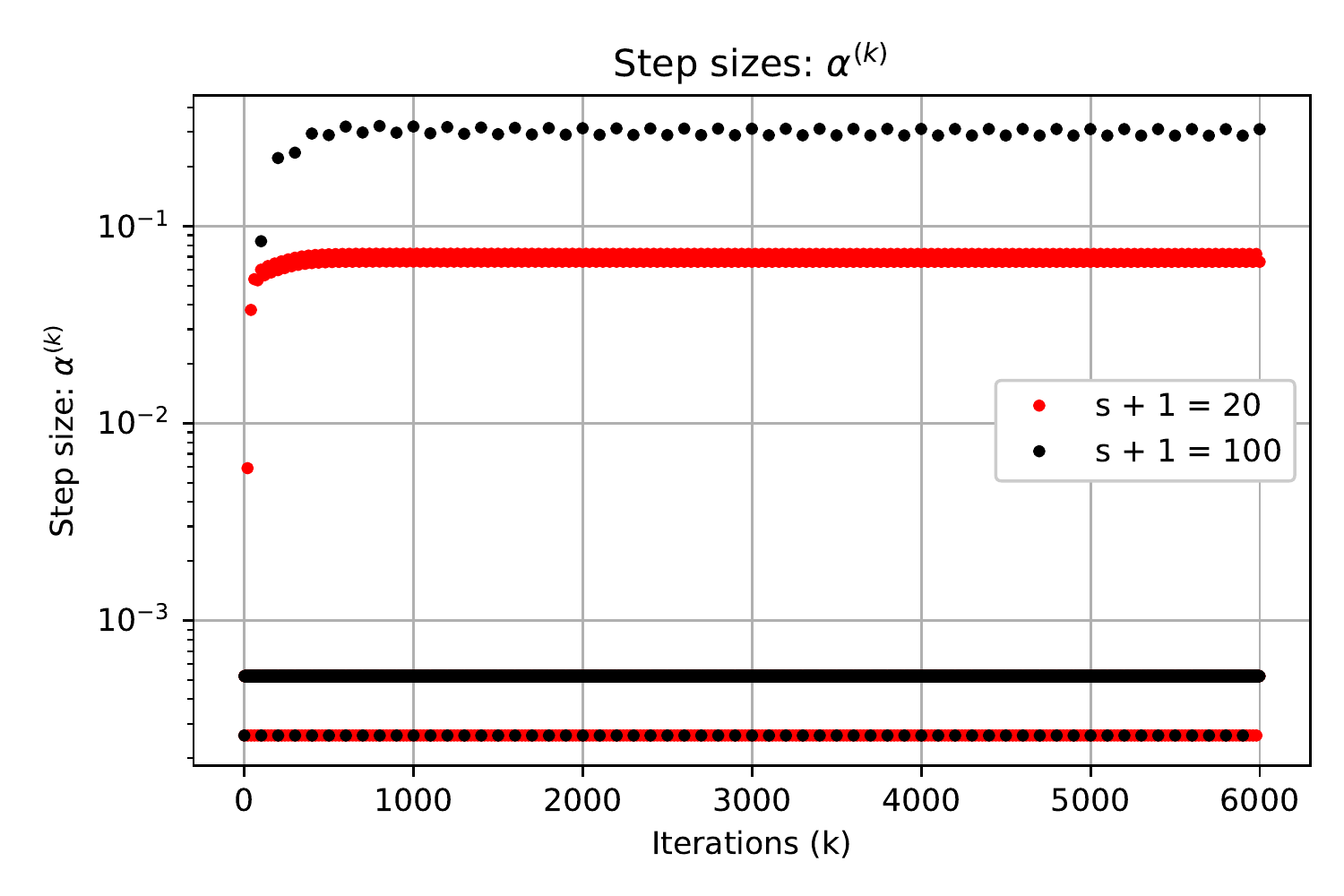}
    \includegraphics[width=0.3\textwidth]{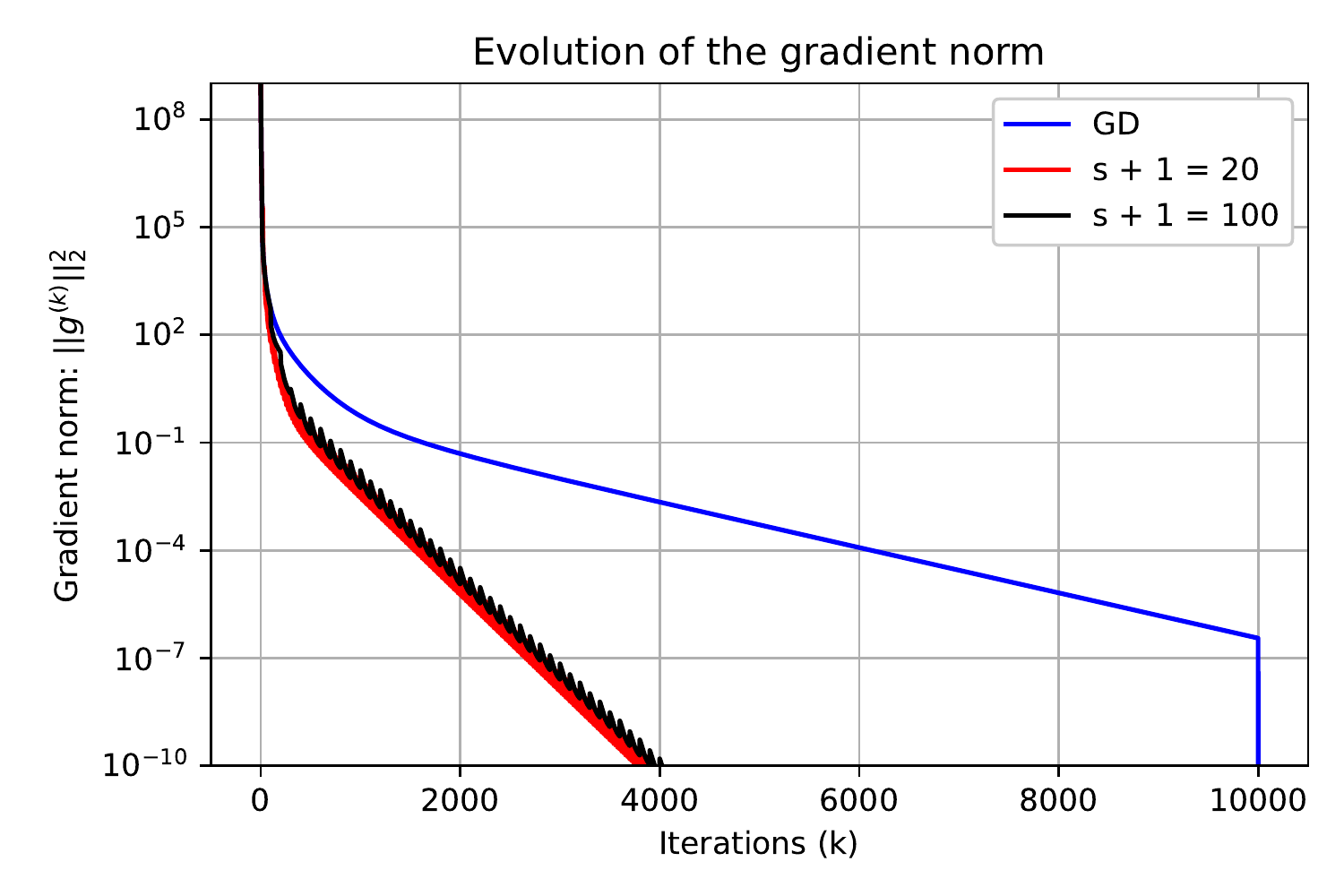}
    
    \caption{Gradient Descent with a kick. The top row corresponds to a positive definite $A$, while the bottom row corresponds to a positive semi-definite $A$.}
    \label{fig:gdeigkick}
\end{figure}

\subsection{Example from \cite{Du2017}}

The work \cite{Du2017} presents an example of a smooth non-convex function $F: \R^2 \to \R$, which is made up quadratic regions that are joined using splines. This provided motivated to investigate the behaviour of the proposed algorithms on this non-quadratic problem. (Specific problem details can be found in \cite{Du2017}.)

Figure~\ref{fig:hard2dProblem} shows the surface and contour plot of the function $F$ together with iterates of Gradient Descent (GD) run with the default step-size as suggested by authors in \cite{Du2017}. To see why this problem is hard, once can investigate the size of gradient during the trajectory $\{x^{(k)}\}_{k=0}^\infty$.

\begin{figure}
    \centering
    \includegraphics[width=0.48\textwidth]{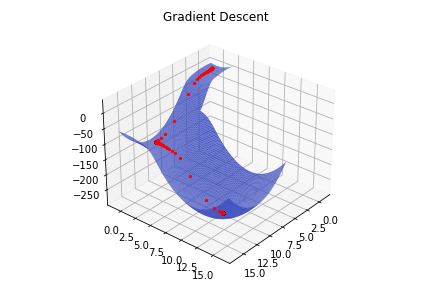}
    \includegraphics[width=0.48\textwidth]{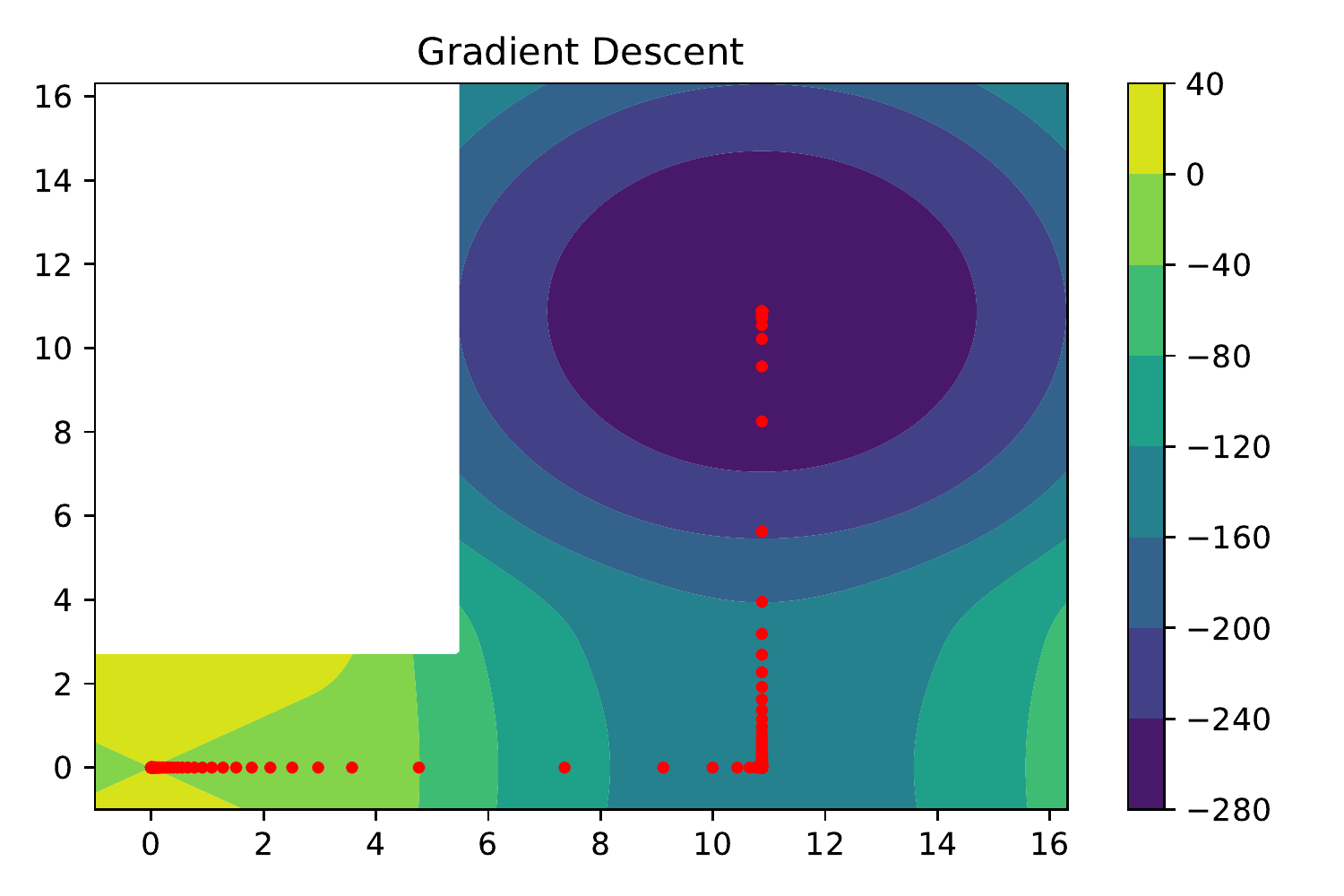}    
    \caption{Surface plot and contour plot of a function $F:\R^2 \to \R$ defined in \cite{Du2017}. We run gradient descent algorithm with a default step-size as suggested in \cite{Du2017} and initial point $x^{(0)} = (0.0001, 0.0001)^T$.}
    \label{fig:hard2dProblem}
\end{figure}

\begin{figure}
    \centering
    \includegraphics[width=0.48\textwidth]{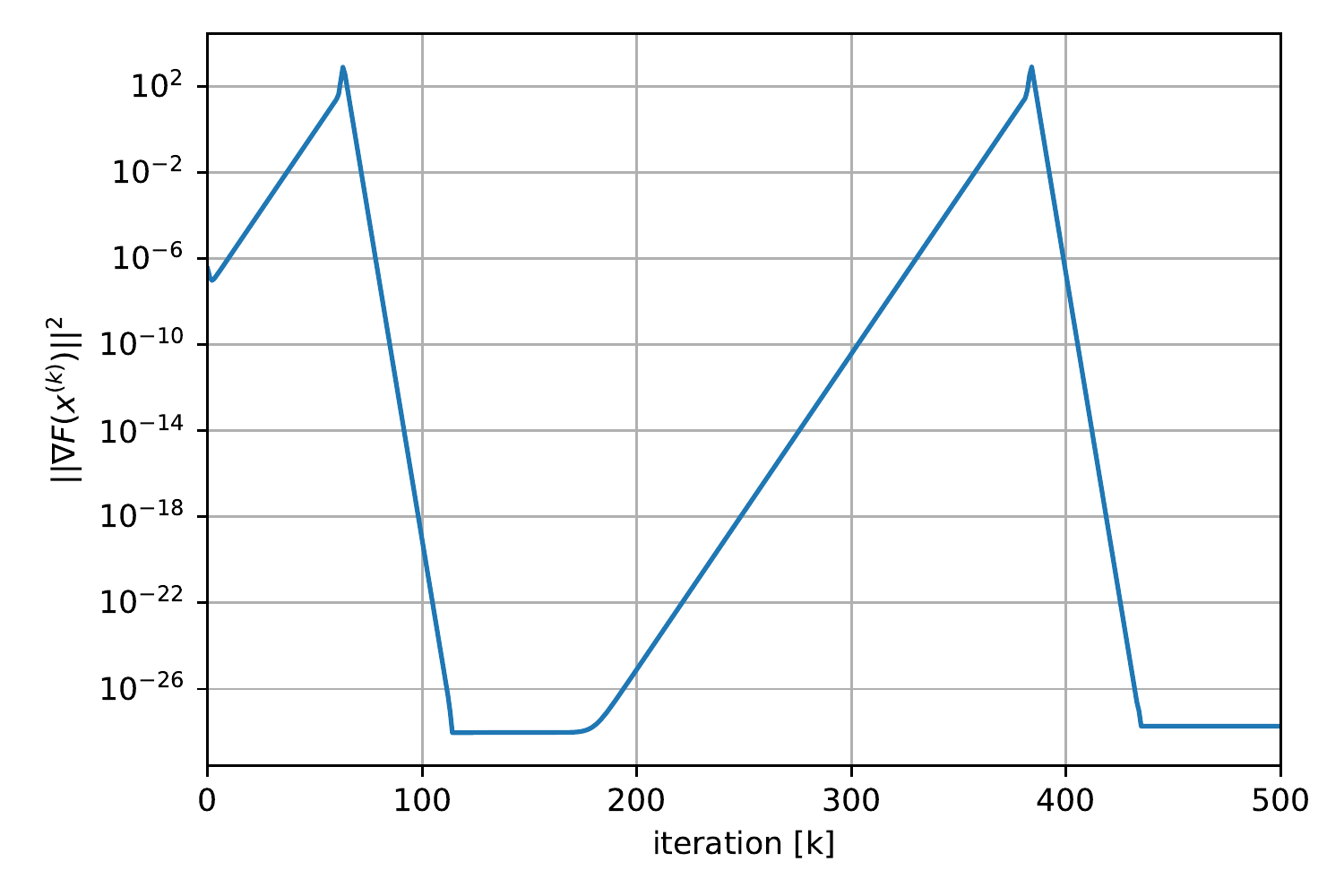}
    \includegraphics[width=0.48\textwidth]{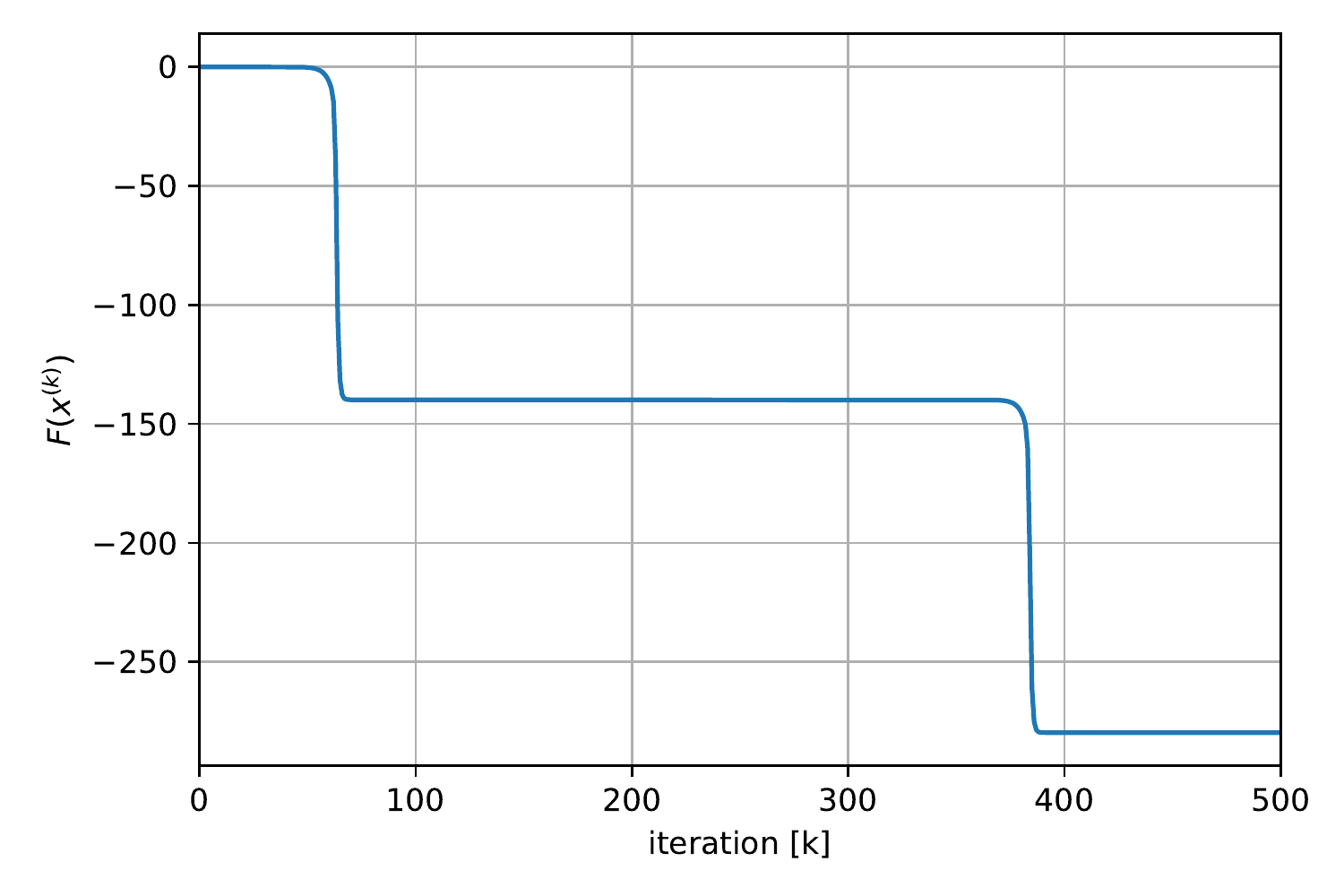}    
    \caption{Evolution of $F(x^{(k)})$ and $\| \nabla F(x^{(k)})\|^2$ of iterates produced by GD algorithm.}
    \label{fig:hard2dProblem_evol}
\end{figure}
In Figure~\ref{fig:hard2dProblem_evol} one observes that at the initial starting point, the norm of the gradient is already very small, and it takes almost 80 iterations to make a significant improvement in the function value $F(\cdot)$. Then, however, the iterates gets trapped close to a saddle point around $(4\cdot e, 0)^T$, and the size of squared gradient becomes close to $10^{-28}$. To escape the saddle point, one suggested approach is to use Perturbed GD (PGD) \cite{jin2017escape}, which perturbs the current iterate in the case when the size of gradient is below some threshold for a specified number of iterations (see \cite{jin2017escape} for details). 
An alternative approach that is investigated now is to utilize GD-EIG-Kick (Algorithm~\ref{alg:GDKick}).
\begin{figure}
    \centering
    \includegraphics[width=0.48\textwidth]{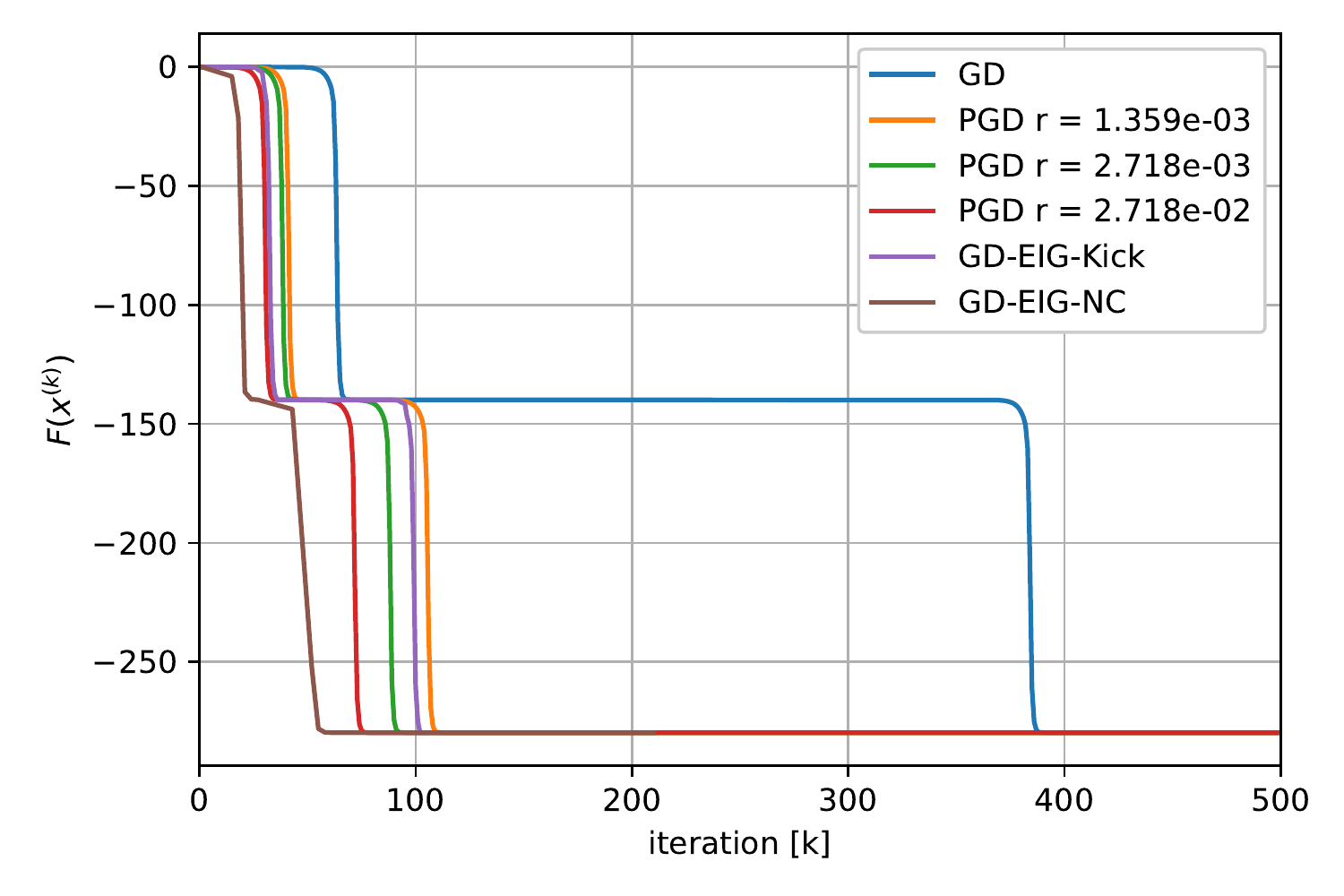}
    \includegraphics[width=0.48\textwidth]{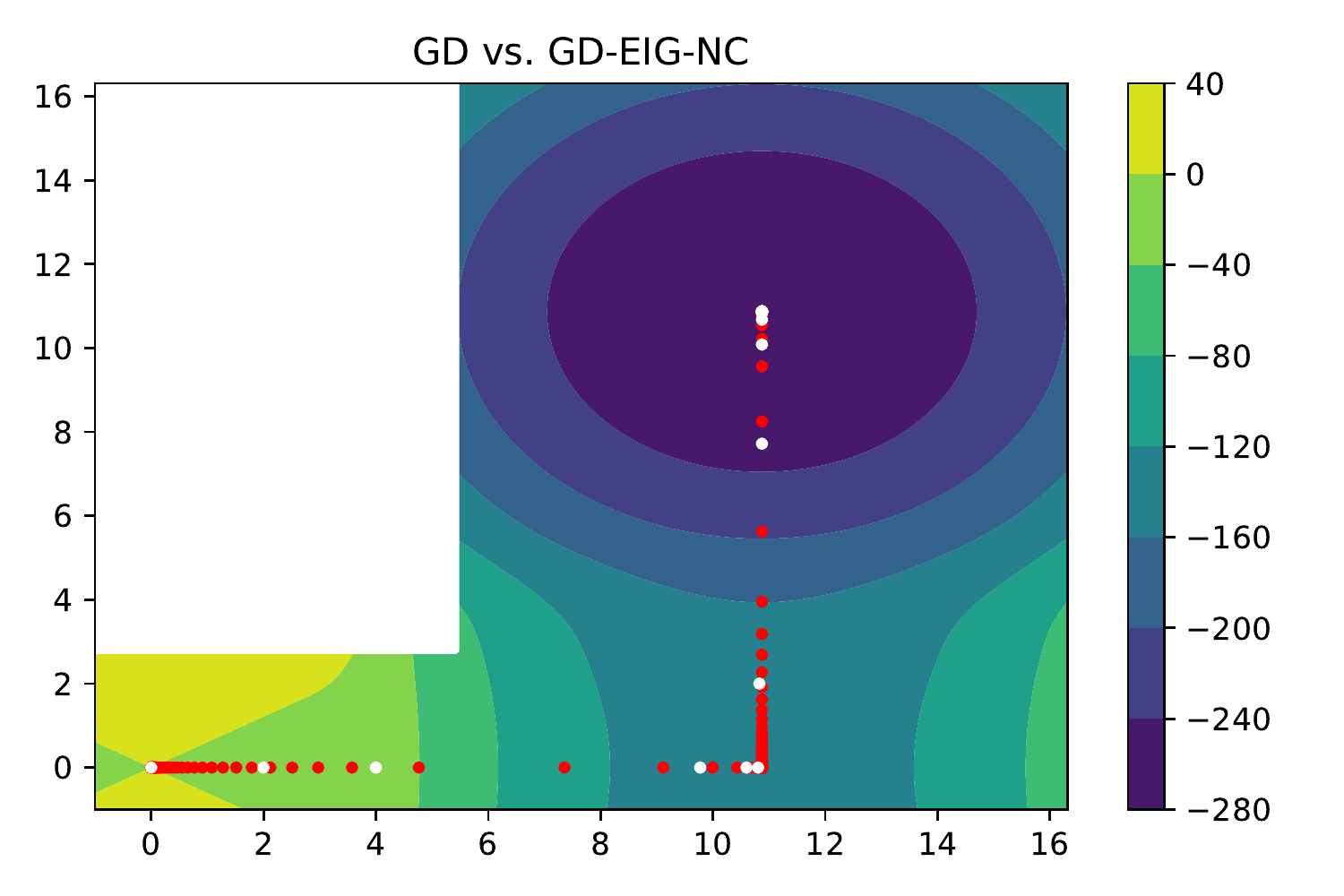}
    \caption{Left: Comparison of GD with PGD for various setting, together with GD-EIG-Kick. Note that for GD-EIG-Kick we count on x-axis the cost, not the iterations itself, the reason is that we need to evaluate multiple hessian-vector products when minimizing the model $m_k(\Delta)$.
    Right: Comparison of GD with GD-EIG-Kick. Note that the GD-EIG-Kick needs only a few iterations. 
    }
    \label{fig:hard2dProblemCOmparison}
\end{figure}
In Figure~\ref{fig:hard2dProblemCOmparison} we compare the evolution of GD, and PGD with various perturbations\footnote{The noise that is added to the current iteration is sampled from sphere with radius $r$. The choice of $r=2.718\cdot10^{-2}$ is a setting used by \cite{Du2017}.}
Note that the performance of PGD is sensitive on the choice of hyper-parameters. On the other hand, GD-EIG-Kick does not require tuning and we simply run it with $s=2$; the performance is comparable to PGD.

We also used GD-EIG-Kick to minimize the quadratic model 
$m_k(\Delta) = F(x_k)+\langle \nabla F(x_k), \Delta\rangle
 + \frac12 \Delta^T \nabla^2 F(x_k) \Delta
 $
that fits the structure of \eqref{probmin}
with $A = \nabla^2 F(x_k)$ and $b = -\nabla F(x_k)$.
If negative curvature is discovered, we would go in that direction, otherwise we would take step towards the approximate solution $\bar \Delta \approx \min_{\Delta} m_k(\Delta)$ that was obtained by GD-EID. We denote this on Figure~\ref{fig:hard2dProblemCOmparison} as GD-EIG-NC. These numerical results support the use of GD-EIG-Kick to help escape saddle points.

\section{Conclusion}
\label{Section_Conclusion}

This work formalized the connection between GD with a fixed step size and the PM, both with and without fixed momentum, when applied to a quadratic function. Thus, GD implicitly provides an approximation to the leftmost eigen-pair of the Hessian.

Several examples from recent literature show that GD with a fixed step size takes exponential time to escape saddle points. These examples were re-visited, to show that if `freely available' eigeninformation was used, the performance of GD with an adaptive step size may be better than previously suggested. In particular, it may be possible to use an adaptive step length based on the estimate of the leftmost eigenvalue to escape saddle points more quickly than if a fixed step size is used. 

The dynamics of the gradient were presented, and its was explained that a step size of $2/\lambda_1$ can be used for GD, as long as the algorithm is initalized carefully. The special case $\R^2$ was discussed, and it was shown that if $\lambda_1$ is known, then a full eigen-decomposition of $A$ is available after 2 iterations of GD with the fixed step size $\alpha = 1/\lambda_1$.

A new algorithm called GD-EIG-Kick was presented, which uses the approximate eigenvalue information to attempt a long step, a `kick', every $s$ iterations. GD-EIG-Kick is guaranteed to converge when $A$ is PD, and it behaves better in  practice than vanilla GD.

Numerical experiments confirmed that GD does provide an estimate of the leftmost eigen-pair in practice, although many iterations are needed for an accurate approximation. Examples were also presented that show the benefits of GD-EIG-Kick compared with vanilla GD,  which supports the view that approximate eigen (curvature) information should be used within GD algorithms where possible.

\subsection{Future work}\label{SectionNonconvexGDEIG}

The work in this paper relies on the fact that the objective function is quadratic. It is more difficult to say whether the ideas may generalize, and investigating this is outside the scope of the current work. However, future work may investigate a Newton-GD type algorithm (akin to the Newton-CG algorithm), where, at a given iterate, a possibly nonconvex quadratic model is built, and GD-EIG is used to find a search direction (possibly of negative curvature).

\bibliographystyle{plain}
\bibliography{refs}

\begin{thebibliography}{10}

\bibitem{Bellavia2013}
Stefania Bellavia, Jacek Gondzio, and Benedetta Morini.
\newblock A matrix-free preconditioner for sparse symmetric positive definite
  systems and least-squares problems.
\newblock {\em SIAM Journal on Scientific Computing}, 35(1):192--211, 2013.

\bibitem{Chen2017}
S.~Chen, S.~Ma, and W.~Liu.
\newblock Geometric descent method for convex composite minimization.
\newblock In {\em Advances in Neural Information Processing Systems 30}, pages
  636--644, 2017.

\bibitem{Demmel97}
James~W. Demmel.
\newblock {\em Applied Numerical Linear Algebra}.
\newblock SIAM, Philadelphia, USA, 1997.

\bibitem{Diakonikolas2019}
J.~Diakonikolas and L.~Orecchia.
\newblock The approximate duality gap technique: A unified theory of
  first-order methods.
\newblock {\em SIAM Journal on Optimization}, 29(1):660--689, 2019.

\bibitem{Drusvyatskiy2018}
D.~Drusvyatskiy, M.~Fazel, and S.~Roy.
\newblock An optimal first order method based on optimal quadratic averaging.
\newblock {\em SIAM Journal on Optimization}, 28(1):251--271, 2018.

\bibitem{Du2017}
Simon~S. Du, Chi Jin, Jason~D. Lee, Michael~I. Jordan, Aarti Singh, and
  Barnabas Poczos.
\newblock Gradient descent can take exponential time to escape saddle points.
\newblock In I.~Guyon, U.V. Luxburg, S.~Bengio, H.~Wallach, R.~Fergus,
  S.~Vishwanathan, and R.~Garnett, editors, {\em Advances in Neural Information
  Processing Systems 30}, 2017.
\newblock 31st Conference on Neural Information Processing Systems (NIPS 2017),
  Long Beach, CA, USA.

\bibitem{Erway2020}
J.B. Erway, J.~Griffin, R.F. Marcia, and R.~Omheni.
\newblock Trust-region algorithms for training responses: machine learning
  methods using indefinite hessian approximations.
\newblock {\em Optimization Methods and Software}, 35(3):460--487, 2020.

\bibitem{Ghadimi2012}
S.~Ghadimi and G.~Lan.
\newblock Optimal stochastic approximation algorithms for strongly convex
  stochastic composite optimization i: A generic algorithmic framework.
\newblock {\em SIAM Journal on Optimization}, 22(4):1469--1492, 2012.

\bibitem{Golub96}
G.~H. Golub and C.~F.~Van Loan.
\newblock {\em Matrix Computations}.
\newblock JHU Press, Baltimore, USA, 3 edition, 1996.

\bibitem{Greenbaum1997}
A.~Greenbaum.
\newblock {\em Iterative Methods for Solving Linear Systems}.
\newblock Frontiers in Applied Mathematics. SIAM, Philadelphia, 1997.

\bibitem{Hestenes1952}
M.~Hestenes and E.~Stiefel.
\newblock Methods of conjugate gradients for solving linear systems.
\newblock {\em Journal of Research of the National Bureau of Standards},
  49(6):409--436, 1952.

\bibitem{Jahani2021}
M.~Jahani, N.V.C Gudapati, C.~Ma, M.~Tak\'a\v{c}, and R.~Tappenden.
\newblock Fast and safe: accelerated gradient methods with optimality
  certificates and underestimate sequences.
\newblock {\em Computational Optimization and Applications}, 79:369--404, 2021.

\bibitem{jin2017escape}
Chi Jin, Rong Ge, Praneeth Netrapalli, Sham~M Kakade, and Michael~I Jordan.
\newblock How to escape saddle points efficiently.
\newblock In {\em International Conference on Machine Learning}, pages
  1724--1732. PMLR, 2017.

\bibitem{Lanczos1950}
C.~Lanczos.
\newblock An iterative method for the solution of the eigenvalue problem of
  linear differential and integral operators.
\newblock {\em Journal of Research of the National Bureau of Standards},
  45:225--282, 1950.

\bibitem{Lanczos1952}
C.~Lanczos.
\newblock Solution of systems of linear equations by minimized iterations.
\newblock {\em Journal of Research of the National Bureau of Standards},
  49:33--53, 1952.

\bibitem{Lee2016}
Jason~D Lee, Max Simchowitz, Michael~I Jordan, and Benjamin Recht.
\newblock Gradient descent only converges to minimizers.
\newblock In {\em Conference on Learning Theory}, pages 1246--1257, 2016.

\bibitem{Meurant2006}
G.~Meurant.
\newblock {\em The Lanczos and Conjugate Gradient Algorithms}.
\newblock SIAM, Philadelphia, 2006.

\bibitem{Nesterov1983}
Yu. Nesterov.
\newblock A method for solving the convex programming problem with convergence
  rate $\oo(1/k^2)$.
\newblock {\em Dokl. Akad. Nauk SSSR}, 269(3):543--547, 1983.

\bibitem{Nesterov2005}
Yu. Nesterov.
\newblock Smooth minimization of non-smooth functions.
\newblock {\em Mathematical Programming}, 103(1):127--152, 2005.

\bibitem{Nesterov2013}
Yu. Nesterov.
\newblock Gradient methods for minimizing composite functions.
\newblock {\em Mathematical Programming}, 140(1):125--161, 2013.

\bibitem{Nesterov2004}
Yurii Nesterov.
\newblock {\em Introductory Lectures on Convex Optimization: A Basic Course}.
\newblock Springer Science and Business Media, 2004.

\bibitem{Nocedal06}
Jorge Nocedal and Stephen~J. Wright.
\newblock {\em Numerical Optimization}.
\newblock Springer Series in Operations Research. Springer, 2 edition, 2006.

\bibitem{Paige1975}
C.C. Paige and M.A. Saunders.
\newblock Solution of sparse indefinite systems of linear equations.
\newblock {\em SIAM Journal on Numerical Analysis}, 12:617--629, 1975.

\bibitem{Paternain2019}
S.~Paternain, A.~Mokhtari, and A.~Ribeiro.
\newblock A newton-based method for nonconvex optimization with fast evasion of
  saddle points.
\newblock {\em SIAM Journal on Optimization}, 29(1):343--368, 2019.

\bibitem{Polyak1964}
B.T. Polyak.
\newblock Some methods of speeding up the convergence of iteration methods.
\newblock {\em USSR Computational Math- ematics and Mathematical Physics},
  4.5:1--17, 1964.

\bibitem{Royer2020}
C.W. Royer, M.~O’Neill, and S.J. Wright.
\newblock A newton-cg algorithm with complexity guarantees for smooth
  unconstrained optimization.
\newblock {\em Mathematical Programming}, 180:451--488, 2020.

\bibitem{Saad2003}
Y.~Saad.
\newblock {\em Iterative Methods for Sparse Linear Systems}.
\newblock SIAM, Philadelphia, 2 edition, 2003.

\bibitem{Trefethen97}
Lloyd~N. Trefethen and III David~Bau.
\newblock {\em Numerical Linear Algebra}.
\newblock SIAM, Philadelphia, USA, 1997.

\bibitem{vanderVorst2003}
H.A. van~der Vorst.
\newblock {\em Iterative Krylov Methods for Large Linear Systems}.
\newblock Cambridge Monographs on Applied and Computational Mathematics.
  Cambridge University Press, Cambridge, 2003.

\bibitem{Xu2018}
P.~Xu, B.~He, C.~De Sa, , I.~Mitliagkas, and C.~R\'{e}.
\newblock Accelerated stochastic power iteration.
\newblock {\em Proceedings of the Twenty-First International Conference on
  Artificial Intelligence and Statistics}, PMLR 84:58--67, 2018.

\end{thebibliography}

\end{document}